\documentclass{amsart}[11pt]

\usepackage[
left=1.25in, right=1.25in]{geometry}
\usepackage[all]{xy}
\usepackage{graphicx}
\usepackage{indentfirst}
\usepackage{bm}
\usepackage{amsthm}
\usepackage{mathrsfs}
\usepackage{latexsym}
\usepackage{amsmath}
\usepackage{amssymb}
\usepackage{color}
\usepackage{hyperref}
\usepackage{microtype}

\usepackage{todonotes}
\newcommand{\sC}{\mathscr{C}}

\newcommand{\sI}{\mathscr{I}}

\begin{document}
\theoremstyle{plain}
\newtheorem{theorem}{Theorem~}[section]
\newtheorem{main}{Main Theorem~}
\newtheorem{lemma}[theorem]{Lemma~}
\newtheorem{proposition}[theorem]{Proposition~}
\newtheorem{corollary}[theorem]{Corollary~}
\newtheorem{definition}[theorem]{Definition~}
\newtheorem{notation}[theorem]{Notation~}
\newtheorem{example}[theorem]{Example~}
\newtheorem{remark}[theorem]{Remark~} 
\newtheorem*{cor}{Corollary~}
\newtheorem*{question}{Question}
\newtheorem*{claim}{Claim}
\newtheorem*{conjecture}{Conjecture~}
\newtheorem*{fact}{Fact~}
\renewcommand{\proofname}{\bf Proof}

\newcommand{\ty}{\tilde{y}}

%\gra{1}
\newcommand{\graza}{
\raisebox{-.3cm}{
\begin{tikzpicture}
\begin{scope}[xscale=.7,yscale=.7]
\draw[thick] (0,0)--(0,1);
\draw[thick] (.5,0) arc (180:0:.25 and .4);
\draw[thick] (.5,1) arc (180:360:.25 and .4);
\end{scope}
\end{tikzpicture}
}
}

%\gra{2}
\newcommand{\grazb}{
\raisebox{-.3cm}{
\begin{tikzpicture}
\begin{scope}[xscale=.7,yscale=.7]
\draw[thick] (1,0)--(1,1);
\draw[thick] (0,0) arc (180:0:.25 and .4);
\draw[thick] (0,1) arc (180:360:.25 and .4);
\end{scope}
\end{tikzpicture}
}
}

%\gra{3}
\newcommand{\grazc}{
\raisebox{-.3cm}{
\begin{tikzpicture}
\begin{scope}[xscale=.7,yscale=.7]
\draw[thick] (0,0)--(0,1);
\draw[thick] (.5,0)--(.5,1);
\draw[thick] (1,0)--(1,1);
\end{scope}
\end{tikzpicture}
}
}

%\gra{4}
\newcommand{\grazd}{
\raisebox{-.3cm}{
\begin{tikzpicture}
\begin{scope}[xscale=.7,yscale=.7]
\draw[thick] (0,0)--(1,1);
\draw[thick] (.5,0) arc (180:0:.25 and .4);
\draw[thick] (0,1) arc (180:360:.25 and .4);
\end{scope}
\end{tikzpicture}
}
}

%\gra{5}
\newcommand{\graze}{
\raisebox{-.3cm}{
\begin{tikzpicture}
\begin{scope}[xscale=.7,yscale=.7]
\draw[thick] (1,0)--(0,1);
\draw[thick] (0,0) arc (180:0:.25 and .4);
\draw[thick] (.5,1) arc (180:360:.25 and .4);
\end{scope}
\end{tikzpicture}
}
}

%\gra{6}
\newcommand{\grazf}{
\raisebox{-.3cm}{
\begin{tikzpicture}
\begin{scope}[xscale=.7,yscale=.7]
\draw[thick] (1,0)--(1,1);
\draw[thick] (0,0)--(.5,1);
\draw[thick] (.5,0)--(0,1);
\node at (0,.5) {\tiny{$R$}};
\end{scope}
\end{tikzpicture}
}
}

%\gra{7}
\newcommand{\grazg}{
\raisebox{-.3cm}{
\begin{tikzpicture}
\begin{scope}[xscale=.7,yscale=.7]
\draw[thick] (.5,0)--(1,1);
\draw[thick] (0,0) arc (180:0:.5 and .4);
\draw[thick] (0,1) arc (180:360:.25 and .4);
\node at (.75,.15) {\tiny{$R$}};
\end{scope}
\end{tikzpicture}
}
}

%\gra{8}
\newcommand{\grazh}{
\raisebox{-.3cm}{
\begin{tikzpicture}
\begin{scope}[xscale=.7,yscale=.7]
\draw[thick] (1,0)--(.5,1);
\draw[thick] (0,0) arc (180:0:.25 and .4);
\draw[thick] (0,1) arc (180:360:.5 and .4);
\node at (.75,.85) {\tiny{$R$}};
\end{scope}
\end{tikzpicture}
}
}

%\gra{9}
\newcommand{\grazi}{
\raisebox{-.3cm}{
\begin{tikzpicture}
\begin{scope}[xscale=.7,yscale=.7]
\draw[thick] (0,0)--(0,1);
\draw[thick] (1,0)--(.5,1);
\draw[thick] (.5,0)--(1,1);
\node at (.75,.15) {\tiny{$R$}};
\end{scope}
\end{tikzpicture}
}
}

%\gra{10}
\newcommand{\grazj}{
\raisebox{-.3cm}{
\begin{tikzpicture}
\begin{scope}[xscale=.7,yscale=.7]
\draw[thick] (0,0)--(.5,1);
\draw[thick] (.5,0) arc (180:0:.25 and .4);
\draw[thick] (0,1) arc (180:360:.5 and .4);
\node at (.65,.85) {\tiny{$R$}};
\end{scope}
\end{tikzpicture}
}
}

%\gra{11}
\newcommand{\grazk}{
\raisebox{-.3cm}{
\begin{tikzpicture}
\begin{scope}[xscale=.7,yscale=.7]
\draw[thick] (.5,0)--(0,1);
\draw[thick] (0,0) arc (180:0:.5 and .4);
\draw[thick] (.5,1) arc (180:360:.25 and .4);
\node at (0,.5) {\tiny{$R$}};
\end{scope}
\end{tikzpicture}
}
}

%\gra{12}
\newcommand{\grazl}{
\raisebox{-.3cm}{
\begin{tikzpicture}
\begin{scope}[xscale=.7,yscale=.7]
\draw[thick] (.5,0)--(.5,1);
\draw[thick] (0,0) arc (180:0:.5 and .4);
\draw[thick] (0,1) arc (180:360:.5 and .4);
\node at (.7,.85) {\tiny{$R$}};
\node at (.7,.15) {\tiny{$R$}};
\end{scope}
\end{tikzpicture}
}
}

%\gra{13}
\newcommand{\grazm}{
\raisebox{-.3cm}{
\begin{tikzpicture}
\begin{scope}[xscale=.7,yscale=.7]
\draw[thick] (0,0)--(1,1);
\draw[thick] (.5,0)--(0,1);
\draw[thick] (1,0)--(.5,1);
\node at (.1,.35) {\tiny{$R$}};
\node at (.75,.95) {\tiny{$R$}};
\end{scope}
\end{tikzpicture}
}
}

%\gra{14}
\newcommand{\grazn}{
\raisebox{-.3cm}{
\begin{tikzpicture}
\begin{scope}[xscale=.7,yscale=.7]
\draw[thick] (1,0)--(0,1);
\draw[thick] (0,0)--(.5,1);
\draw[thick] (.5,0)--(1,1);
\node at (.1,.65) {\tiny{$R$}};
\node at (.75,.05) {\tiny{$R$}};
\end{scope}
\end{tikzpicture}
}
}

%\gra{15}
\newcommand{\grazo}{
\raisebox{-.3cm}{
\begin{tikzpicture}
\begin{scope}[xscale=.7,yscale=.7]
\draw[thick] (0,0)--(1,1);
\draw[thick] (1,0)--(0,1);
\draw[thick] (.3,0)--(.3,1);
\node at (.1,.65) {\tiny{$R$}};
\node at (.45,.15) {\tiny{$R$}};
\node at (.55,.75) {\tiny{$R$}};
\end{scope}
\end{tikzpicture}
}
}

%\gra{16}
\newcommand{\grazp}{
\raisebox{-.3cm}{
\begin{tikzpicture}
\begin{scope}[xscale=.7,yscale=.7]
\draw[thick] (0,0)--(1,1);
\draw[thick] (1,0)--(0,1);
\draw[thick] (.6,0)--(.6,1);
\node at (.1,.5) {\tiny{$R$}};
\node at (.75,.05) {\tiny{$R$}};
\node at (.75,.95) {\tiny{$R$}};
\end{scope}
\end{tikzpicture}
}
}

%\gra{21}
\newcommand{\graba}{
\raisebox{-.3cm}{
\begin{tikzpicture}
\begin{scope}[xscale=.7,yscale=.7]
\draw[thick] (0,0)--(0,1);
\draw[thick] (2/3,0)--(2/3,1);
\end{scope}
\end{tikzpicture}
}
}

%\gra{22}
\newcommand{\grabb}{
\raisebox{-.3cm}{
\begin{tikzpicture}
\begin{scope}[xscale=.7,yscale=.7]
\draw[thick] (0,0) arc (180:0:1/3);
\draw[thick] (0,1) arc (180:360:1/3);
\end{scope}
\end{tikzpicture}
}
}

%\gra{23}
\newcommand{\grabc}{
\raisebox{-.3cm}{
\begin{tikzpicture}
\begin{scope}[xscale=.7,yscale=.7]
\draw[thick] (0,0)--(.8,1);
\draw[thick] (.8,0)--(0,1);
\node at (.1,1/2) {\tiny{$R$}};
\end{scope}
\end{tikzpicture}
}
}

\title{The Grothendieck Ring of a Family of Spherical Categories}

\author{Zhengwei Liu}
\address[Z. Liu]{Yau Mathematical Science Center and Department of Mathematics, Tsinghua University, Beijing, 100084, China}
\email{liuzhengwei@mail.tsinghua.edu.cn}

\author{Christopher Ryba}
\address[C. Ryba]{Department of Mathematics, Massachusetts Institute of Technology, Cambridge, MA 02139, USA}
\email{ryba@mit.edu}

%-----------------------------------------------------------------------------------
%ABSTRACT AND KEYWORDS
%-----------------------------------------------------------------------------------

%\date{\today}

\maketitle

% Keywords
%\hspace*{3,6mm}\textit{Keywords:}
%MSC: 18M20, 46L37, 81T40, 81R50.

\begin{abstract}
The first author constructed a $q$-parameterized spherical category $\sC$ over $\mathbb{C}(q)$ in \cite{LiuYB}, whose simple objects are labelled by all Young diagrams. In this paper, we compute closed-form expressions for the fusion rule of $\sC$, using Littlewood-Richardson coefficients, as well as the characters (including a generating function), using symmetric functions with infinite variables. 
\end{abstract}

\section{Introduction}
Jones introduced planar algebras in \cite{Jon99} inspired by subfactor theory and knot theory. The topological notion of (spherical) planar algebra is parallel to the algebraic notion of (spherical) pivotal, monoidal category. The first author investigated the skein-theoretical classification of planar algebras and discovered a continuous family of unshaded planar algebras over $\mathbb{C} $ from the classification of Yang-Baxter relation planar algebras in \cite{LiuYB}. 
This family was constructed in terms of $q$-parameterized generators and relations in linear skein theory.
This family could be regarded as a planar algebra $\sC$ over $\mathbb{C}(q)$ with a generic parameter $q$.
The canonical idempotent category of this unshaded planar algebra $\sC$ is a $\mathbb{Z}_2$-graded spherical monoidal category. 
It was shown in \cite{LiuYB} that the Grothendick ring $\mathcal{G}$ of $\sC$  has simple objects $X_{\lambda}$ labelled by all Young diagrams $\lambda$ with an explicit construction of a minimal idempotent $\ty_{\lambda}$ in $\sC$ which represents $X_{\lambda}$.

The main purpose of this paper is computing the fusion rule of $\mathcal{G}$,
\begin{align}\label{Equ: fusion rule}
X_{\mu}X_{\nu}=\sum_{\lambda} R_{\mu, \nu}^ \lambda X_{\lambda},  
\end{align}
where $R_{\mu, \nu}^ \lambda \in \mathbb{N}$ is called the fusion coefficient for Young diagrams $\mu$, $\nu$, $\lambda$.
We compute the fusion coefficient $R_{\mu, \nu}^ \lambda$ in a closed-form expression in Theorem \ref{Thm: Main3},
\begin{align}
R_{\mu, \nu}^ \lambda &= \sum_{\alpha, \beta, \gamma} c_{\alpha, \beta}^\mu c_{\beta', \gamma}^\nu c_{\alpha, \gamma}^\lambda,
\end{align}
where $\gamma^\prime$ is the Young diagram dual to $\gamma$ and $c_{\cdot,\cdot}^{\cdot}$ is the Littlewood-Richardson coefficient.

The paper is organized as follows.
In \S \ref{Sec: Yang-Baxter relation planar algebras and spherical categories}, we recall some basic properties of $\sC$ and its type $A$ Hecke subalgebra $H$. We prove in Theorem \ref{Thm: ring iso X} that the Grothendieck ring  $\mathcal{G}$ of $\sC$ is the polynomial ring freely generated by the fundamental representations $\{X_{1^n} : n\in\mathbb{N} \}$, where $1^n$ is the Young diagram with one column and $n$ cells. In particular,  $\mathcal{G}$ is commutative.

In \S \ref{Sec: Fusion Rules of Fundamental Representations for the Generic Case}, 
we compute the fusion rule of $\mathcal{G}$ with respect to the fundamental representations in Theorem \ref{Thm: fusion with column}:
\begin{align}
X_{(1^r)} X_{\mu}&=\sum_{i=0}^{r} \sum_{\nu \in \mu- i} \sum_{\lambda \in \nu+1^{r-i}}X_{\lambda}.
\end{align}
The multiplicity of $X_{\lambda}$ is the number of ways of constructing $\lambda$ from $\mu$ by removing $i$ cells, no two in the same column, and then adding $r-i$ cells, no two in the same row.
The proof follows from an explicit construction of the basis of $\hom(X_{(1^r)} X_{\mu}, X_{\lambda})$ in $\sC$ through the linear skein theory of the Yang-Baxter relation planar algebra.

In \S \ref{Sec: Characters, Generating Functions and Fusion Rules for the Generic Case}, we compute fusion rules of $\mathcal{G}$. In principle, one can compute the fusion rule of $\mathcal{G}$ recursively using the fusion rule of fundamental representations. However, the complexity grows exponentially w.r.t. the size of the Young diagrams. 
We observe that $\mathcal{G}$ is isomorphic to the ring $\Lambda$ of symmetric polynomial with infinite variables. We establish a ring isomorphism $\Phi: \mathcal{G} \to \Lambda$ in Definition \ref{Def: Iso} and consider $\Phi(X_\lambda)$ as the character of the simple object $X_{\lambda}$ of $\mathcal{G}$.
%(It is the unique ring isomorphism such that the restriction of $\Phi$ on the minimal idempotent $y_{\lambda}$ of the Hecke subalgebra $H$ is the Schur polynomial $s_{\lambda}$ for any Young diagram $\lambda$.)
We prove in Theorem \ref{Thm: Main1} that
\begin{align}\label{Equ: Weyl Character}
\Phi(X_\lambda) & = \sum_{\mu} (-1)^{|\mu|} s_{\lambda/2\mu} \;. 
\end{align}
where $s_{\lambda/2\mu}$ is a skew-Schur polynomial.
Moreover, we compute the generating function of the characters in closed form in Theorem \ref{Thm: Main2},
\begin{align}
\sum_{\lambda} s_\lambda(x) \Phi(X_\lambda)(y) &= \prod_{i_1 \leq i_2}\frac{1}{1 + x_ix_j} \prod_{i, j}\frac{1}{1-x_i y_j}.
\end{align}
Using the generating function, we compute the fusion coefficient in a closed form, namely Equation \eqref{Equ: fusion rule}, in Theorem \ref{Thm: Main3}. 
Our computational tools on the characters and the generating function come from the theory of symmetric functions \cite{Macdonald}, which we recall in \S \ref{Sec: Characters, Generating Functions and Fusion Rules for the Generic Case}.

In this paper, we compute the fusion rule of $\mathscr{C}$ over $\mathbb{C}(q)$.
Unitary fusion categories $\mathscr{C}^{N,k,\ell}$, $N,k,\ell \in \mathbb{N}$, were constructed in \cite{LiuYB} as quotients of $\mathscr{C}$ over $\mathbb{C}$ at $q=e^{\frac{\pi i}{2N+2}}$. In particular, $\mathscr{C}^{N,0,1}$ is an exceptional quantum subgroup of $SU(N)_{N+2}$, conjectured to be isomorphic to the exceptional quantum subgroup constructed by Xu in \cite{Xu98} in 1998 through the $\alpha$-induction of the conformal inclusion $SU(N)_{N+2} \subseteq SU(N(N+1)/2)_1$. Xu asked the question to compute the fusion rules of these exceptional quantum subgroups \cite{Xu}, which we will compute in the future.
From this point of view, $\mathscr{C}$ can be regarded as the parameterization of a family of exceptional quantum subgroups. It was conjecture in \cite{LiuYB} that there is a continuous family of monoidal categories parameterizing the exceptional quantum subgroups from the $\alpha$-induction of any family of conformal inclusions of quantum groups. We believe that our methods in this paper also apply to the other continuous families of monoidal categories, if they exist.

\subsection*{Acknowledgements}
Zhengwei Liu was partially supported by Grant 100301004 from Tsinghua University and by Templeton Religion Trust under the grant TRT 159.
Zhengwei Liu would like to thank Pavel Etingof and Feng Xu for helpful discussions and to thank Arthur Jaffe for the hospitality at Harvard University. 
Christopher Ryba would like to thank Pavel Etingof for useful conversations.

\section{Yang-Baxter relation planar algebras and spherical categories}\label{Sec: Yang-Baxter relation planar algebras and spherical categories}
The first author constructed the following continuous family of Yang-Baxter relation planar algebras $\sC$ in terms of generators and relations in linear skein theory in \cite{LiuYB}.

\begin{definition}\label{Def:Centralizer algebra}
Let $\mathscr{C}_{\bullet}$ be the unshaded planar algebra over $\mathbb{C}(q)$ with circle parameter
$$\raisebox{-.3cm}{
\begin{tikzpicture}
\begin{scope}[xscale=.6,yscale=.6]
\draw[thick] (0,0) arc (0:360:.5);
\end{scope}
\end{tikzpicture}
}=\delta =q+q^{-1} \;, $$
which is generated by
$R=\raisebox{-.3cm}{
\begin{tikzpicture}
\begin{scope}[xscale=.6,yscale=.6]
\draw[thick] (0,0)--(1,1);
\draw[thick] (1,0)--(0,1);
\node at (0,1/2) {\tiny{$R$}};
\end{scope}
\end{tikzpicture}
}$
with the following relations:

\begin{align}\label{YBrelation1}
\raisebox{-.3cm}{
\begin{tikzpicture}
\begin{scope}[xscale=.6,yscale=.6]
\draw[thick] (0,0)--(1,1);
\draw[thick] (1,0)--(0,1);
\node at (0,1/2) {\tiny{$R$}};
\end{scope}
\end{tikzpicture}
}&=i
\raisebox{-.3cm}{
\begin{tikzpicture}
\begin{scope}[xscale=.6,yscale=.6]
\draw[thick] (0,0)--(1,1);
\draw[thick] (1,0)--(0,1);
\node at (1/2,1) {\tiny{$R$}};
\end{scope}
\end{tikzpicture}
}, \\\label{YBrelation2}
\raisebox{-.3cm}{
\begin{tikzpicture}
\begin{scope}[xscale=.6,yscale=.6]
\draw[thick] (0,0)--(2/3,2/3);
\draw[thick] (2/3,2/3) arc (135:-135: 1.414/6);
\draw[thick] (0,1)--(2/3,1/3);
\node at (0,1/2) {\tiny{$R$}};
\end{scope}
\end{tikzpicture}
}&=0, \\\label{YBrelation3}
\raisebox{-.3cm}{
\begin{tikzpicture}
\begin{scope}[xscale=.6,yscale=.4]
\draw[thick] (0,0)--(1,1)--(0,2);
\draw[thick] (1,0)--(0,1)--(1,2);
\node at (0,1/2) {\tiny{$R$}};
\node at (0,1+1/2) {\tiny{$R$}};
\end{scope}
\end{tikzpicture}
}
&= \graba
-\frac{1}{\delta} \grabb
, \\
\label{YBrelation4}
\grazo&=\frac{i}{\delta^2}(\grazf+\grazg+\grazh)-\frac{1}{\delta^2}(\grazi+\grazj+\grazk)+i\grazp.
\end{align}
\end{definition}
The vector space  $\sC_{n}$ consists of linear sums of $R$-labelled planar diagrams with $2n$ boundary points modulo the above relations.
Consider a disc with  $2n$ boundary points numbered by $\{1,2,\ldots, 2n\}$ clockwise.
A pairing $p$ of $\{1,2,\ldots, 2n\}$ is a bijection on $\{1,2,\ldots, 2n\}$, such that $p^2$ is the identity and $p(i) \neq i$, $\forall 1\leq i \leq 2n$.
We call $\{i,p(i)\}$ a pair of the paring $p$. Let $P_{n}$ be the set of pairings of $2n$ boundary points.

We can construct a diagram in the disc which connects the $n$ pairs of boundary points by $n$ strings with a minimal number of crossings.
(The minimal condition is equivalent to that any two strings either intersect at one point transversally or do not intersect.)
Such diagrams have been used by Brauer to construct the Brauer algebras.
We label each crossing of the diagram by the generator $R$, then we obtain an element in $\sC_{n}$, denoted by $\hat{p}$. Note that there are four choices to label $R$ at each crossing, and the corresponding elements in $\sC_{n}$ differ by a phase due to Relation \eqref{YBrelation1}. We fix a choice at the beginning to define $\hat{p}$.

\begin{proposition}\label{Prop: Bp Basis}
The set $\mathcal{B}_n=\{\hat{p} : p\in P_{n}\}$ is a basis of the vector space $\sC_{n}$ over $\mathbb{C}(q)$.
\end{proposition}

\begin{proof}
Applying the Yang-Baxter relation, any element in $\sC_{n}$ is a linear sum of such $\hat{p}$'s. On the other hand, $\dim\sC_n=(2n-1)!!$ by Corollary 6.6 in \cite{LiuYB}, and $\#\{\hat{p} : p\in P_{n}\}=(2n-1)!!$. Therefore,
$\{\hat{p} : p\in P_{n}\}$ is a basis of $\sC_{n}$.
\end{proof}

\begin{remark}
Note that $\grazo$ and $\grazp$ correspond to the same pairing. When we define the element $\hat{p}$, we fix a choice. Either $\grazo$ or $\grazp$ with the 14 lower terms
\begin{align*}
& \graza,\grazb,\grazc,\grazd,\graze, \\
& \grazf,\grazg,\grazh,\grazi,\grazj,\grazk, \\
& \grazl,\grazm,\grazn,
\end{align*}
form a basis of $\sC_{3}$.
\end{remark}

A planar algebra canonically has a vertical multiplication, a horizontal tensor product and a Markov trace, see \cite{Jon99}. 
For a diagram in $\sC_{n}$, we draw the first $n$ boundary points on the top and the last $n$ boundary points at the bottom. Then $\sC_{n}$ forms an algebra whose multiplication is gluing the diagrams vertically.
The tensor product $\otimes: \sC_n \otimes \sC_m \to \sC_{n+m}$ is a horizontal union of two diagrams. In particular, the algebra $\sC_{n}$ can be embedded in $\sC_{n+1}$ by adding a through string on the right. So $\sC_{\bullet}$ is a filtered algebra.

The planar algebra $\sC_{\bullet}$ has a has a type $A$ Hecke subalgebra $H_{\bullet}$ generated by
$$\alpha=
\frac{q-q^{-1}}{2}\graba+\frac{q-q^{-1}}{2i}\grabb+ \frac{q+q^{-1}}{2} \grabc.$$
The generic type $A$ Hecke algebra has two parameters $q$ and $r$. Here $qr=\sqrt{-1}$. 
The planar algebra has a Markov trace by gluing the top and the bottom from the right. The Markov trace of $\sC_{\bullet}$ extends the Markov trace of the Hecke algebra in \cite{Jon87}.

The idempotent $e=\frac{1}{\delta}\grabb$ in $\sC_2$ is called the Jones idempotent \footnote{It was called the Jones projection is the operator algebraic setting.}.
The two-sided ideal of $\sC_n$ generated by $e$ is denoted by $\sI_n$, called the basic construction ideal. The complement of the maximal idempotent of $\sI_n$ is denoted by $s_n$.
Then, $s_n$ is central in  $\sC_n$ and
\begin{align}
\label{Equ: sn} x s_n&=0, ~\forall~ x \in \sI_n. \\
\label{Equ: sn0} s_n(s_m \otimes s_{n-m}) &=s_n,  ~\forall~ m \leq n.
\end{align}

The following proposition is a consequence of Theorem 6.5 in \cite{LiuYB}:

\begin{proposition}\label{Prop: sn}
For any $n \geq 0$,
\begin{align}
\label{Equ: sn1} H_n &\cong s_nH_n =s_n \sC_n \;, \\
\label{Equ: sn2} \sC_n &=\sI_n \oplus H_n = \sI_n \oplus s_n H_n\;.
\end{align}
\end{proposition}
For each Young diagram $\lambda$, let $|\lambda|=n$ be the number of cells of $\lambda$.
A minimal idempotent $y_{\lambda}$ in $\hom_{\sC}(X^{n}, X^{n})$ was constructed in Section 2.5 in \cite{LiuYB}. These $y_{\lambda}$'s are representatives of the equivalent minimal idempotents of $H_{n}$.
Furthermore, $\ty_{\lambda}=s_n y_{\lambda}$ are representatives of the equivalent minimal idempotents of $s_nH_{n}$.
We refer the readers to \cite{LiuYB} for the explicit construction of $s_n$, $y_{\lambda}$ and $\ty_{\lambda}$.
Following the construction in Theorem 6.5 in \cite{LiuYB}, we have that
\begin{equation}\label{Equ: y=ty}
y_{(1^n)}=\ty_{(1^n)},
\end{equation}
where $(1^n)$ is the Young diagram with one column and $n$ cells.
The Young diagram with one row and $n$ cells is denoted by $(n)$.
We write them as $1^n$ and $n$ for short, if there is no risk of confusion.

We recall the above properties of $\sC$, which we will apply in this paper. We refer the readers to \cite{LiuYB} for the construction of $s_n$, $y_{\lambda}$ and $\ty_{\lambda}$, which we do not repeat here.

\begin{proposition}\label{Prop: iso}
Note that  $\hom_H  (y_{\mu} \otimes y_{\nu}, y_{\lambda} ) \subseteq H_n \subseteq \sC_n$, when $|\mu|+|\nu|=|\lambda|=n$. We have that
\[
s_n \hom_{H}  (y_{\mu} \otimes y_{\nu}, y_{\lambda} ) =\hom_{\sC}  (\ty_{\mu} \otimes \ty_{\nu}, \ty_{\lambda} ) .
\]
\end{proposition}

\begin{proof}
Note that $s_n y_{\mu} \otimes y_{\nu} =\ty_{\mu} \otimes \ty_{\nu}$ and $s_n y_{\lambda}=\ty_{\lambda}$, so
\[
s_n \hom_{H}  (y_{\mu} \otimes y_{\nu}, y_{\lambda} ) \subseteq \hom_{\sC}  (\ty_{\mu} \otimes \ty_{\nu}, \ty_{\lambda} ) .
\]
On the other hand, by Equation \ref{Equ: sn2}, for any element $x$ in $\hom_{\sC}  (\ty_{\mu} \otimes \ty_{\nu}, \ty_{\lambda} ) \subseteq \sC_n$, we have a unique decomposition $x=y+z$, such that $y\in \sI_n$ and $z \in H_n$. Note that
\begin{align*}
y_{\lambda} x (y_{\mu} \otimes y_{\nu}) &=x, \\
y_{\lambda} y (y_{\mu} \otimes y_{\nu})& \in \sI_n ,\\
y_{\lambda} z (y_{\mu} \otimes y_{\nu}) &\in H_n .\\
\end{align*}
So $y_{\lambda} z (y_{\mu} \otimes y_{\nu})=z$, and $z \in \hom_{H}  (y_{\mu} \otimes y_{\nu}, y_{\lambda} )$.
Moreover, $s_n z = s_n x =x$. So
\[
s_n \hom_{H}  (y_{\mu} \otimes y_{\nu}, y_{\lambda} ) = \hom_{\sC}  (\ty_{\mu} \otimes \ty_{\nu}, \ty_{\lambda} ) .
\]
\end{proof}

From the semi-simple, spherical, unshaded planar algebra $\sC$, we obtain an $\mathbb{N}$-graded monoidal category $\sC^{q,0}$ whose degree $n$ objects are idempotents in $\sC_n$.
In subfactor theory, we usually consider the Jones idempotent to be equivalent to the unit in $\sC_0$, namely the empty diagram $\emptyset$. The isometries between the two idempotents are given by the the diagrams $\cap$ and $\cup$. Modulo this relation, we obtain a $\mathbb{Z}_2$-graded spherical category $\sC^{q,1}$, the canonical one associated with the spherical planar algebra $\sC$.

\begin{notation}\label{Not: Y}
Let $\mathcal{G}$ be the Grothendieck ring of $\sC^{q,1}$. It has a basis $X_{\lambda}$ corresponding to the minimal idempotents $\ty_{\lambda}$ of $s_n\sC_n$, $n\in \mathbb{N}$, for all Young diagrams $\lambda$.
\end{notation}

Let $1^r$ be the Young diagram with one column and $r$ cells.
In particular, $X=X_1$, corresponds to the Young diagram with one cell. Then the identity map $1_X$ is a through string, and
\begin{align*}
(\cup \otimes 1_X)(1_X \otimes \cap) &=1_X \; , \\
(1_X \otimes \cup)(\cap \otimes 1_X) &=1_X \; .
\end{align*}
The morphism space $\hom_{\sC}(X^{n}, X^{m})$ consists of linear combinations of $R$-labelled planar diagrams in $\sC$ with $n$ boundary points on the top and $m$ boundary points at the bottom.
For Young diagrams $\mu, \nu, \lambda$, the morphisms of $\sC^{q,1}$ are given by
\[
\hom_{\sC}(X_{\mu} \otimes X_{\nu},  X_{\lambda})= \ty_{\lambda} (\hom_{\sC}(X^{|\mu|+|\nu|}, X^{|\lambda|} )  (\ty_{\mu} \otimes \ty_{\nu}) \;.
\]

\begin{notation}\label{Not: Y}
Let $Y_{\lambda}$ be the element of $\mathcal{G}$ corresponding to the idempotent $y_{\lambda}$.
Note that $\ty_{\lambda}=s_{|\lambda|} y_{\lambda}$, so $y_{\lambda}-\ty_{\lambda}$ is an idempotent in $\sI_{|\lambda|}$. Therefore,
\begin{align}\label{Equ: Y=X+}
Y_{\lambda}=X_{\lambda}+\sum_{|\mu| < |\lambda|} n_{\lambda, \mu} X_{\mu}, \text{ for some } n_{\lambda, \mu} \in \mathbb{N}.
\end{align}
We call $n_{\lambda, \mu}$ the {\bf extended constants}.
Then we can solve for the $ X_{\lambda}$ recursively in terms of the $Y_{\lambda}$, and
\begin{align}\label{Equ: X=Y+}
X_{\lambda}=Y_{\lambda}+\sum_{|\mu| < |\lambda|} z_{\lambda, \mu} Y_{\mu}, \text{ for some } z_{\lambda, \mu} \in \mathbb{Z}.
\end{align}
We call $z_{\lambda, \mu}$ the {\bf inverse extended constants}.
By Equation \eqref{Equ: y=ty}, for any $n\geq 0$,
\begin{align}\label{Equ: Yn=Xn}
X_{1^n}=Y_{1^n}.
\end{align}
\end{notation}

\begin{theorem}\label{Thm: ring iso X}
The Grothendieck ring  $\mathcal{G}$ of $\sC$ is the polynomial ring in the generators $\{ X_{1^n}: n > 0 \}$. In particular,  $\mathcal{G}$ is commutative.
\end{theorem}

\begin{proof}
Note that $\{X_\lambda\}$ forms a basis of the Grothendieck ring $\mathcal{G}$.
By Equations \eqref{Equ: Y=X+} and \eqref{Equ: X=Y+}, $\{Y_\lambda\}$ also forms a basis of $\mathcal{G}$.
It is known that the set $\{Y_{\lambda} \}$ is a basis of the polynomial ring in the generators $\{Y_{1^n}: n > 0 \}$.
By Equation \eqref{Equ: Yn=Xn}, $\mathcal{G}$ is the polynomial ring in the generators $\{ X_{1^n}: n > 0 \}$.
\end{proof}

Based on the algebraic structure on $\sC_n \cong \hom_{\sC}(X^n,X^n)$,
We decompose the partitions $P_n$ into two subset sets $I_n$ and $T_n$.
A pairing is in $I_n$ if there is a pair among the first $n$ points. On the other hand, a pairing $p$ in $T_n$ pairs the first $n$ points with the last $n$ points.
For any pairing $p \in T_n$, we can identify $p$ with an element $p'$ in the permutation group $S_n$, such that $p'(i)=2n+1-p(i)$, for $1\leq i \leq n$.

\begin{proposition}\label{Prop: IS}
For a pairing $p \in P_n$, we have that $p\in I_n$ iff $\hat{p} \in \sI_n$.
Moreover, $\{s_n \hat{p} : p \in T_n\}$ is a basis of $s_nH_n$.
\end{proposition}

\begin{proof}
Obviously if $p \in P_n$, then $\hat{p} \in \sI_n$ and $s_n \hat{p}=0$.
By Equation \eqref{Equ: sn2}, $\{s_n \hat{p} : p \in T_n\}$ is a spanning set of $s_n H_n$.
By Equation \eqref{Equ: sn1}, $\dim s_n H_n= \dim H_n = \#T_n$, so $\{s_n \hat{p} : p \in T_n\}$ is a basis of $s_nH_n$, and for any $p \in T_n$, $s_n \hat{p} \neq 0$, namely $\hat{p} \notin \sI_n$.
\end{proof}

We define $T_{i,n-i}$ to be a subset of $P_n$ as follows: $T_{0,n}=T_{n,0}=T_n$, and for any $1\leq i \leq n-1$,
\begin{align*}
T_{i,n-i}&=\{p \in P_n: 2n+1-j  \leq p(j) \leq 2n, ~\forall~ 1\leq j \leq i; n+1  \leq p(j) \leq 2n-i, ~\forall~ i+1 \leq j \leq n\} . \\
\end{align*}
We can consider $T_{i,n-i}$ as $T_i \times T_{n-i}$. Note that $T_{i,n-i} \subseteq T_n$.

\begin{notation}
For any Young diagram $\lambda$, $|\lambda|=n$, we express $\ty_{\lambda}$ in terms of the basis $\mathcal{B}_n$,
\[
\ty_{\lambda}=\sum_{p\in P_n} c_p \hat{p} \;, ~c_p\in \mathbb{C}(q).
\]
We define
\begin{align}\label{Equ: yi}
\ty_{\lambda,i}=\sum_{p\in T_{i,n-i}} c_p \hat{p}.
\end{align}
\end{notation}

\begin{lemma}
For any Young diagram $\lambda$, we have that $\ty_{\lambda,0} \ty_{\lambda}=\ty_{\lambda} \ty_{\lambda,0}=\ty_{\lambda}$.
\end{lemma}

\begin{proof}
Note that $s_n \hat{p}=0$, $\forall p \in I_n$, and $\ty_{\lambda}=s_n\ty_{\lambda}$. So
\begin{align*}
\ty_{\lambda} \ty_{\lambda,0}  &=\ty_{\lambda} s_n \sum_{p\in T_{n}} c_p \hat{p}
=\ty_{\lambda} s_n \sum_{p\in P_{n}} c_p \hat{p}
=\ty_{\lambda} s_n \ty_{\lambda}
=\ty_{\lambda} .
\end{align*}
Similarly, $\ty_{\lambda,0} \ty_{\lambda}=\ty_{\lambda}$.
\end{proof}

\begin{lemma}\label{Lem: non-zero}
For any $0\leq i \leq n$, we have that
\[
\ty_{1^n,i}  (\ty_{1^i}\otimes \ty_{1^{n-i}})= c \ty_{1^i}\otimes \ty_{1^{n-i}},
\]
for some $c\neq 0$ in $\mathbb{C}(q)$, and $\displaystyle \lim_{q \to 1} c= {n\choose i}^{-1}$.
\end{lemma}

\begin{proof}
By Proposition \ref{Prop: sn}, $ \ty_{1^k}$ is a minimal central projection in $\sC_k$, so for any $p \in T_{i,n-i}$,
$\hat{p}  (\ty_{1^i} \otimes \ty_{1^{n-i}})$ is a multiple of $\ty_{1^i}\otimes \ty_{1^{n-i}}$. Therefore,
\[
\ty_{1^n,i}  (\ty_{1^i} \otimes \ty_{1^{n-i}})= c \ty_{1^i}\otimes \ty_{1^{n-i}},
\]
for some $c \in \mathbb{C}(q)$. Moreover,
$\ty_{1^n,i} \ty_{1^n}= c \ty_{1^n}$. We need to show that $c\neq 0$.
For any $p \in T_n$, we consider $p$ as a permutation.
Without loss of generality, we assume that the strings of $\hat{p}$ move vertically and the generator $R$'s of $\hat{p}$ are all labelled on the left side of the crossings. Note that $R \ty_{1^2}=-\ty_{1^2}$, so
\[
\hat{p} \ty_{1^n} = (-1)^{|p|} \ty_{1^n},
\]
where $|p|$ is the number of the crossing $R$'s in $\hat{p}$.
We express $\ty_{1^n}$, $y_{1^n}$ and $y_{1^i} \otimes  y_{1^{n-i}}$ in terms of the basis $\mathcal{B}_n$ as
\begin{align*}
\ty_{1^n}&=\sum_{p\in P_n} c_p \hat{p} \;, \\
y_{1^n}&=\sum_{p\in P_n} c'_p \hat{p} \;.
\end{align*}
Then
\[
c =\sum_{p\in T_{i,n-i}} \frac{(-1)^{|p|}}{n!}  c_p  \;.
\]
Recall that $\ty_{\lambda}=s_n y_{\lambda}$, so by Proposition \ref{Prop: IS},
\[
c_p=c'_p, ~\forall~ p\in T_n.
\]

When $q \to 1$, the Hecke algebra $H$ specializes to the symmetric group algebra; the generator $\alpha$ becomes the symmetric braiding; $\alpha-R \to 0$; and
$n! y_{1^n} $ becomes the alternating sum of permutations of the symmetric groups $S_n$.
So for any  $p\in T_n$,
\begin{align}
\lim_{q \to 1} c'_p = \frac{(-1)^{|p|}}{n!} .
\end{align}
Then
\begin{align}
\lim_{q \to 1} c =\lim_{q \to 1} \sum_{p\in T_{i,n-1}} \frac{(-1)^{|p|}}{n!}  c_p = \sum_{p\in T_{i,n-i}} \frac{1}{n!}={n\choose i}^{-1}.
\end{align}
Therefore, $C \neq 0$ in $\mathbb{C}(q)$.
\end{proof}

\section{Fusion Rules of Fundamental Representations for the Generic Case}\label{Sec: Fusion Rules of Fundamental Representations for the Generic Case}
In this section, we compute the fusion rule for $X_{1^n} \otimes $ in the Grothendieck ring $\mathcal{G}$, and construct a basis for the hom space. We apply this to study the characters of the simple objects using symmetric functions. Recall that $\{X_\lambda\}$ indexed by Young diagrams are a basis of $\mathcal{G}$. The {\bf structure constants} of the multiplication $R_{\mu, \nu}^{\lambda}$ are defined by:
\[
X_\mu X_\nu = \sum_\lambda R_{\mu, \nu}^\lambda X_\lambda
\]

\begin{notation}
We define the morphism $\cup_n$ in $\hom_{\sC}(X^{2n},\emptyset)$ as
\[
\cup_n=
\raisebox{-.5cm}{
\begin{tikzpicture}
\draw[thick] (4,2) arc (-180:0:1);
\node at (3.8,1.8) {n};
\end{tikzpicture}}
=
\raisebox{-.5cm}{
\begin{tikzpicture}
\draw (4,2) arc (-180:0:.5  and .25);
\draw (3,2) arc (-180:0:1.5 and .75);
\node at (3.5,1.8) {$\cdots$};
\node at (3,2.2) {1};
\node at (4,2.2) {n};
\node at (5,2.2) {n+1};
\node at (6,2.2) {2n};
\end{tikzpicture}}
\]
where the label $n$ in the first picture indicates the number of parallel strings.
\end{notation}
By Proposition 9.2 in \cite{LiuYB}, the dual object (or the $180^{\circ}$ rotation) of $\ty_{1^n}$ is $\ty_{n}$.
In particular, $\cup_n(\ty_{1^n} \otimes \ty_{n} )$ is a non-zero morphism in $\hom_{\sC}(\ty_{n} \otimes \ty_{1^n},\emptyset)$.

\begin{proposition}\label{Prop: dual object}
The dual object of $\ty_{\lambda}$ is $\ty_{\lambda'}$, where $\lambda'$ is the reflection of $\lambda$ in the diagonal, called the Young diagram dual to $\lambda$.
\end{proposition}

\begin{proof}
This is a consequence of Proposition 9.6 in \cite{LiuYB}.
We give a quick proof here. The duality map $\lambda \to \lambda'$ is a $\mathbb{Z}_2$ automorphism of the principal graph of the planar algebra $\sC$, which is Young's lattice. This $\mathbb{Z}_2$ fixes the the Young diagrams $\emptyset$ and $1$, and switches $1^2$ and $2$. Therefore it has to be the reflection in the diagonal.
\end{proof}

\begin{notation}
For any Young diagram $\lambda$, we define the following sets of Young diagrams:
\begin{enumerate}
\item
$\lambda-1^n$ are Young diagrams that removes $n$ cells from $\lambda$, and no two cells in the same row;
\item
$\lambda+1^n$ are Young diagrams that adds $n$ cells to $\lambda$, and no two cells in the same row;
\item
$\lambda-n$ are Young diagrams that removes $n$ cells from $\lambda$, and no two cells in the same column;
\item
$\lambda+n$ are Young diagrams that adds $n$ cells to $\lambda$, and no two cells in the same column.
\end{enumerate}
\end{notation}

The following result is well-known for the type $A$ Hecke algebra. It can be derived from the fusion rule of fundamental representations of (quantum) $SU(N)$, as $N \to \infty$. The fusion rule can be characterized by Schur polynomials.
\begin{lemma}\label{Lem: fusion 1}
Suppose $\lambda$ and $\mu$ are Young diagrams. If $n=|\mu|-|\lambda|\geq 0$, then
\begin{align*}
\dim \hom_{H}(y_{\lambda} \otimes y_{1^{n}}, y_{\mu})&=\left\{
\begin{aligned}
1, &  ~\forall~ \mu \in \lambda+1^n; \\
0, & ~\forall~ \mu \notin \lambda+1^n.
\end{aligned}
\right.
\end{align*}
\end{lemma}
We give an explicit construction of a non-zero morphism $\rho$ in $\hom_{H}(y_{\lambda} \otimes y_{1^{n}}, y_{\mu})$.

\begin{lemma}\label{Lem: dim 1}
Suppose $\lambda$ and $\mu$ are Young diagrams. If $n=|\mu|-|\lambda|\geq 0$, then
\begin{align*}
\dim \hom_{\sC}(\ty_{\lambda} \otimes \ty_{1^{n}}, \ty_{\mu})&=
\left\{
\begin{aligned}
1, &  ~\forall~ \mu \in \lambda+1^n; \\
0, & ~\forall~ \mu \notin \lambda+1^n.
\end{aligned}
\right. \\
\dim \hom_{\sC}(\ty_{\lambda}  \otimes \ty_{n} , \ty_{\mu})&=
\left\{
\begin{aligned}
1, &  ~\forall~ \mu \in \lambda+n; \\
0, & ~\forall~ \mu \notin \lambda+n;
\end{aligned}
\right.
\end{align*}
If $n=|\lambda|-|\mu|\geq 0$, then
\begin{align*}
\dim \hom_{\sC}(\ty_{\lambda}  \otimes \ty_{1^n} , \ty_{\mu})&=
\left\{
\begin{aligned}
1, &  ~\forall~ \mu \in \lambda-n; \\
0, & ~\forall~ \mu \notin \lambda-n;
\end{aligned}
\right. \\
\dim \hom_{\sC}(\ty_{\lambda}  \otimes \ty_{n} , \ty_{\mu})&=
\left\{
\begin{aligned}
1, &  ~\forall~ \mu \in \lambda-1^n; \\
0, & ~\forall~ \mu \notin \lambda-1^n.
\end{aligned}
\right.
\end{align*}
\end{lemma}

\begin{proof}
If $n=|\mu|-|\lambda|\geq 0$,
then by Equation \eqref{Equ: sn1}, Proposition \ref{Prop: iso} and Lemma \eqref{Lem: fusion 1}, we have
\begin{align*}
\dim \hom_{\sC}(\ty_{\lambda} \otimes \ty_{1^{n}}, \ty_{\mu})&= \dim \hom_{H}(y_{\lambda} \otimes y_{1^{n}}, y_{\mu}) =
\left\{
\begin{aligned}
1, &  ~\forall~ \mu \in \lambda+1^n; \\
0, & ~\forall~ \mu \notin \lambda+1^n.
\end{aligned}
\right.
\end{align*}

The planar algebra $\sC$ has a $\mathbb{Z}_2$ automorphism $\Omega$ mapping the generator $R$ to $-R$. By Proposition 9.5 in \cite{LiuYB}, the idempotent $\Omega(\ty_{\lambda})$ is equivalent to $\ty_{\lambda^\prime}$. (The dual Young diagram $\lambda'$ is denoted by $\Omega(\lambda)$ in \cite{LiuYB}.)

In particular, $n^\prime=1^n$.
Note that $\mu \in \lambda+n$ iff $\mu^\prime \in \lambda^\prime+1^n$. So
\begin{align*}
\dim \hom_{\sC}(\ty_{\lambda}  \otimes \ty_{n} , \ty_{\mu})&= \dim \hom_{\sC}(\ty_{\lambda^\prime}  \otimes \ty_{1^n} , \ty_{\mu^\prime}) =
\left\{
\begin{aligned}
1, &  ~\forall~ \mu \in \lambda+n; \\
0, & ~\forall~ \mu \notin \lambda+n.
\end{aligned}
\right.
\end{align*}

By Proposition 9.2 in \cite{LiuYB}, the dual object (or $180^{\circ}$ rotation) of $\ty_{1^n}$ is $\ty_{n}$.
If $n=|\lambda|-|\mu|\geq 0$, then by Frobenius reciprocity,
\begin{align*}
\dim \hom_{\sC}(\ty_{\lambda}  \otimes \ty_{1^n} , \ty_{\mu})&= \dim \hom_{\sC}(\ty_{\lambda}, \ty_{\mu} \otimes \ty_{n} )=
\left\{
\begin{aligned}
1, &  ~\forall~ \mu \in \lambda-n; \\
0, & ~\forall~ \mu \notin \lambda-n;
\end{aligned}
\right.
\end{align*}
\begin{align*}
\dim \hom_{\sC}(\ty_{\lambda}  \otimes \ty_{n} , \ty_{\mu})= \dim \hom_{\sC}(\ty_{\lambda}, \ty_{\mu} \otimes \ty_{1^n} )=
\left\{
\begin{aligned}
1, &  ~\forall~ \mu \in \lambda-1^n; \\
0, & ~\forall~ \mu \notin \lambda-1^n.
\end{aligned}
\right.
\end{align*}

\end{proof}

\begin{notation}
Suppose $a,b,c \in \mathbb{N}$, and $n=a+b+c$.
Let $p_{a,b,c}\in P_n$ be the pairing
\begin{align}\label{Equ: p_{a,b,c}}
p_{a,b,c}(k)=\left\{
\begin{aligned}
&(2n+1-k), &&\forall~ 1\leq k \leq a \text{ or } 2n-a< k \leq 2n ; \\
&(2a+2b+1-k), &&\forall~ a < k \leq a+2b; \\
&(2n+2b+1-k), &&\forall~ n+b < k \leq 2n-a; \\
\end{aligned}
\right.
\end{align}
We can identify $\hat{p}_{a,b,c} \in \sC_{n}$ as a morphism in $\hom_{\sC}(X^{a+b} \otimes X^{b+c}, X^{a+c})$, illustrated as
\[
\hat{p}_{a,b,c}=\raisebox{-1cm}{
\begin{tikzpicture}
\draw[thick] (0,0) --++(0,2);
\draw[thick] (1,2) arc (-180:0:.5);
\draw[thick] (1,0) --++(2,2);
\node at (-.2,1) {a};
\node at (1-.2,1.8) {b};
\node at (1.7,1) {c};
\end{tikzpicture}}
=
\raisebox{-1.5cm}{
\begin{tikzpicture}
\draw (1,0) --++(0,2);
\draw (2,0) --++(0,2);
\draw (3.25,0) --++(4,2);
\draw (4.5,0) --++(4,2);
\draw (4,2) arc (-180:0:.5  and .25);
\draw (3,2) arc (-180:0:1.5 and .75);
\node at (1.5,1) {$\cdots$};
\node at (6,1) {$\cdots$};
\node at (3.5,1.8) {$\cdots$};
\node at (1,2.2) {1};
\node at (2,2.2) {a};
\node at (3,2.2) {a+1};
\node at (4,2.2) {a+b};
\node at (5,2.2) {a+b+1};
\node at (6,2.2) {a+2b};
\node at (7.25,2.2) {a+2b+1};
\node at (8.5,2.2) {n+b};
\node at (1,-.2) {2n};
\node at (2,-.2) {2n+1-a};
\node at (3.25,-.2) {2n-a};
\node at (4.5,-.2) {n+b+1};
\end{tikzpicture}}
\]
where $a,b,c$ in the first picture indicate the number of parallel strings.
\end{notation}

\begin{notation}
Suppose $\mu$ is a Young diagram, $|\mu|=a+b$. Take Young diagrams $\nu \in \mu-b$ and $\lambda \in \nu +1^{c}$.
By Lemma \ref{Lem: dim 1}, there are non-zero morphisms $\rho_{1,\nu} \in \hom_{\sC}(\ty_{\mu}, \ty_{\nu} \otimes \ty_{b} )$ and $\rho_{2,\nu} \in \hom_{\sC}(\ty_{\nu} \otimes \ty_{1^{c}}, \ty_{\lambda})$.
We construct a morphism $\rho'_{\mu,\nu,\lambda} \in \hom_{\sC}(\ty_{\mu} \otimes \ty_{1^b} \otimes \ty_{1^{c}}, \ty_{\lambda})$ as
\begin{align}\label{Equ: rho}
\rho'_{\mu,\nu,\lambda}&:=\rho_{2,\nu} \hat{p}_{a,b,c} (\rho_{1,\nu} \otimes \ty_{1^b} \otimes \ty_{1^{c}}) \;.
\end{align}
We identify $\ty_{b+c}$ with a morphism in $\hom_{\sC}(\ty_{1^{b+c}}, \ty_{1^{b}} \otimes \ty_{1^{c}} )$ and construct a morphism $\rho'_{\mu,\nu,\lambda} \in \hom_{\sC}(\ty_{\mu} \otimes \ty_{1^{b+c}}, \ty_{\lambda})$:
\begin{align}\label{Equ: rho'}
\rho_{\mu,\nu,\lambda}&:=\rho'_{\mu,\nu,\lambda}(\ty_{\mu} \otimes \ty_{b+c})  =\rho_{2,\nu} \hat{p}_{a,b,c} (\rho_{1,\nu} \otimes \ty_{b+c}) \;.
\end{align}

Their pictorial representations are
\begin{align*}
\rho'_{\mu,\nu,\lambda}=
\raisebox{-2.5cm}{
\begin{tikzpicture}
\draw[thick] (0,0) --++(0,2);
\draw[thick] (1,2) arc (-180:0:.5);
\draw[thick] (1,0) --++(2,2);
\node at (-.2,1) {$\ty_{\nu}$};
\node at (1-.2,1.8) {$\ty_{b}$};
\node at (1.7,3.2) {$\ty_{1^{b}}$};
\node at (2.7,3.2) {$\ty_{1^{c}}$};
\node at (0,3.2) {$\ty_{\mu}$};
\node at (0,2.5) {$\rho_{1,\nu}$};
\node at (0,-1.2) {$\ty_{\lambda}$};
\node at (0,-.5) {$\rho_{2,\nu}$};
\draw[thick] (2,2) --++(0,1.5);
\draw[thick] (3,2) --++(0,1.5);
\draw[thick] (0,2) --(.5,2.5)--++(0,1);
\draw[thick] (1,2) --(.5,2.5);
\draw[thick] (0,0) --(.5,-.5)--++(0,-1);
\draw[thick] (1,0) --(.5,-.5);
\end{tikzpicture}
}
\quad \text{and} \quad
\rho_{\mu,\nu,\lambda}=
\raisebox{-2.5cm}{
\begin{tikzpicture}
\draw[thick] (0,0) --++(0,2);
\draw[thick] (1,2) arc (-180:0:.5);
\draw[thick] (1,0) --++(2,2);
\node at (-.2,1) {$\ty_{\nu}$};
\node at (1-.2,1.8) {$\ty_{b}$};
\node at (1.7,1.8) {$\ty_{1^b}$};
\node at (1.7,1) {$\ty_{1^c}$};
\node at (2,2.5) {$\ty_{1^{b+c}}$};
\node at (0,3.2) {$\ty_{\mu}$};
\node at (0,2.5) {$\rho_{1,\nu}$};
\node at (0,-1.2) {$\ty_{\lambda}$};
\node at (0,-.5) {$\rho_{2,\nu}$};
\draw[thick] (2,2) --(2.5,2.5)--++(0,1);
\draw[thick] (3,2) --(2.5,2.5);
\draw[thick] (0,2) --(.5,2.5)--++(0,1);
\draw[thick] (1,2) --(.5,2.5);
\draw[thick] (0,0) --(.5,-.5)--++(0,-1);
\draw[thick] (1,0) --(.5,-.5);
\end{tikzpicture}
}
\end{align*}
\end{notation}

\begin{lemma}\label{Lem: Basis 0}
Suppose $a,b,c \in \mathbb{N}$ and $r=b+c$. For any Young diagrams $\mu$ and $\lambda$, $|\mu|=a+b$, $|\lambda|=a+c$,
the elements $\{\rho'_{\mu,\nu, \lambda} : \nu \in \mu-b, \nu\in \lambda -1^{c} \}$ are linearly independent in $\hom_{\sC}(\ty_{\mu} \otimes \ty_{1^b} \otimes \ty_{1^c}, \ty_{\lambda})$.
\end{lemma}

\begin{proof}
By Frobenius reciprocity, for any $\nu$, $(\ty_{\nu} \otimes \cup_{b} ) (\rho_{\mu,\nu} \otimes \ty_{1^b})\neq 0$ in $\hom_{\sC}(\ty_{\mu} \otimes \ty_{1^b}, \ty_{\nu} )$.
As $\sC$ is semi-simple, there is a morphism $\rho_{3,\nu} \in \hom_{\sC}(\ty_{\nu}, \ty_{\mu} \otimes \ty_{1^b})$, such that
\[
(\ty_{\nu} \otimes \cup_{b} ) (\rho_{\mu,\nu} \otimes \ty_{1^b}) \rho_{3,\nu}=\ty_{\nu}.
\]
If
\[
\sum_{\nu \in \mu-b, \nu\in \lambda -1^{c}} c_{\nu} \rho'_{\mu,\nu, \lambda}=0,  ~c_{\nu} \in \mathbb{C}(q),
\]
then for any $\nu' \in \mu-b, \nu'\in \lambda -1^{c}$,
\[
\rho_{3,\nu'} \sum_{\nu \in \mu-b, \nu\in \lambda -1^{c}} \rho'_{\mu,\nu, \lambda}=c_{\nu'} \rho_{2,\nu'}=0.
\]
So $c_{\nu'}=0$. Therefore, $\{\rho'_{\mu,\nu, \lambda} : \nu \in \mu-b, \nu\in \lambda -1^{c} \}$ are linearly independent.

\end{proof}

\begin{lemma}\label{Lem: Basis 1}
Suppose $a,b,c \in \mathbb{N}$ and $r=b+c$. For any Young diagrams $\mu$ and $\lambda$ with $|\mu|=a+b$, $|\lambda|=a+c$,
the morphisms $\{\rho_{\mu,\nu, \lambda} : \nu \in \mu-b, \nu\in \lambda -1^{c} \}$ form a spanning set of  $\hom_{\sC}(\ty_{\mu} \otimes \ty_{1^r}, \ty_{\lambda})$.
\end{lemma}

\begin{proof}

For any $p_1 \in P_{a+b}$, $p_2 \in P_{b+c}$ and $p_3 \in P_{a+c}$.
we define
\begin{equation}\label{Equ: xp}
x_{p_1,p_2,p_3}=\hat{p}_3 \hat{p}_{a,b,c} (\hat{p}_1 \otimes \hat{p}_2) .
\end{equation}
By Proposition \ref{Prop: Bp Basis}, $\{x_{p_1,p_2,p_3} : p_1 \in P_{a+b}, p_2 \in P_{b+c}, p_3 \in P_{a+c}.\}$ is a spanning set of $\sC_{n}$, because any pairing in $P_n$ can be implemented by some diagram $x_{p_1,p_2,p_3}$ with a minimal number of crossings.
Note that $y_{1^{b+c}}$ is a central minimal idempotent in $\sC_{b+c}$.
By Equation \eqref{Equ: sn1}, $\ty_{1^{b+c}}$ is a central minimal idempotent in $H_{b+c}$.

By Lemma \ref{Lem: dim 1}, $\dim \hom_{\sC}(\ty_{1^{b+c}}, \ty_{1^b}\otimes \ty_{1^c})=1$, so $\ty_{1^{b+c}} \in \hom_{\sC}(\ty_{1^{b+c}}, \ty_{1^b}\otimes \ty_{1^c})$ and
$(\ty_{1^b}\otimes \ty_{1^c}) \ty_{1^{b+c}}=\ty_{1^{b+c}}$.
We define
\[
\tilde{x}_{p_1,p_2,p_3}=\ty_{\lambda} x_{p_1,p_2,p_3} (\ty_{\mu}\otimes \ty_{1^{b+c}}).
\]
Then $\{\tilde{x}_{p_1,p_2,p_3} : p_1 \in P_{a+b}, p_2 \in P_{b+c}, p_3 \in P_{a+c}.\}$ is a spanning set of $\hom_{\sC}(\ty_{\mu} \otimes \ty_{1^r}, \ty_{\lambda})$.
Recall that the $180^{\circ}$ rotation of $\ty_{1^b}$ is $\ty_{b}$. So
\[
\tilde{x}_{p_1,p_2,p_3}=
\raisebox{-2.5cm}{
\begin{tikzpicture}
\draw[thick] (0,0) --++(0,2);
\draw[thick] (1,2) arc (-180:0:.5);
\draw[thick] (1,0) --++(2,2);
\node at (-.2,1) {$s_a$};
\node at (1-.2,1.8) {$\ty_{b}$};
\node at (1.7,1.8) {$\ty_{1^b}$};
\node at (1.7,1) {$\ty_{1^c}$};
\node at (2,3.2) {$\ty_{1^{b+c}}$};
\node at (2,2.5) {$\hat{p}_2$};
\node at (0,3.2) {$\ty_{\mu}$};
\node at (0,2.5) {$\hat{p}_1$};
\node at (0,-1.2) {$\ty_{\lambda}$};
\node at (0,-.5) {$\hat{p}_3$};
\draw[thick] (2,2) --(2.5,2.5)--++(0,1);
\draw[thick] (3,2) --(2.5,2.5);
\draw[thick] (0,2) --(.5,2.5)--++(0,1);
\draw[thick] (1,2) --(.5,2.5);
\draw[thick] (0,0) --(.5,-.5)--++(0,-1);
\draw[thick] (1,0) --(.5,-.5);
\end{tikzpicture}
}
=\sum_{|\mu|=a} \sum_{j}
\raisebox{-2.5cm}{
\begin{tikzpicture}
\draw[thick] (0,0) --++(0,2);
\draw[thick] (1,2) arc (-180:0:.5);
\draw[thick] (1,0) --++(2,2);
\node at (-.2,1) {$\ty_{\nu}$};
\node at (1-.2,1.8) {$\ty_{b}$};
\node at (1.7,1.8) {$\ty_{1^b}$};
\node at (1.7,1) {$\ty_{1^c}$};
\node at (2,3.2) {$\ty_{1^{b+c}}$};
\node at (2,2.5) {$\ty_{\hat{p}_2}$};
\node at (0,3.2) {$\ty_{\mu}$};
\node at (0,2.5) {$\rho_{1,j}$};
\node at (0,-1.2) {$\ty_{\lambda}$};
\node at (0,-.5) {$\rho_{2,j}$};
\draw[thick] (2,2) --(2.5,2.5)--++(0,1);
\draw[thick] (3,2) --(2.5,2.5);
\draw[thick] (0,2) --(.5,2.5)--++(0,1);
\draw[thick] (1,2) --(.5,2.5);
\draw[thick] (0,0) --(.5,-.5)--++(0,-1);
\draw[thick] (1,0) --(.5,-.5);
\end{tikzpicture}
}
=\sum_{\nu \in \mu-b, \nu\in \lambda -1^{c} } \sum_{j} c_j \rho_{\mu,\nu,\lambda}
,\]
for some $\rho_{1,j} \in \hom_{\ty_{\mu}, \ty_{\nu}\otimes \ty_{b}}$, $\rho_{2,j} \in \hom_{\ty_{\nu} \otimes \ty_{1^c}, \ty_{\lambda} }$, and $c_j \in \mathbb{C}(q)$.
Precisely, the label $a$ is replaced by $s_a$ in the first equality by Equation \eqref{Equ: sn0}.
Then $s_a$ is replaced by $\ty_{\nu}$ in the second equality by Equation~\eqref{Equ: sn1}.
Then we obtain the third equality by Lemma \ref{Lem: dim 1}.
Therefore, $\{\rho_{\mu,\nu, \lambda} : \nu \in \mu-b, \nu\in \lambda -1^{c} \}$ is a spanning set of  $\hom_{\sC}(\ty_{\mu} \otimes \ty_{1^r}, \ty_{\lambda})$.
\end{proof}

\begin{lemma}\label{Lem: Basis 2}
Suppose $a,b,c \in \mathbb{N}$ and $r=b+c$. For any Young diagrams $\mu$ and $\lambda$ with $|\mu|=a+b$, $|\lambda|=a+c$,
the morphisms  $\{\rho_{\mu,\nu, \lambda} : \nu \in \mu-b, \nu\in \lambda -1^{c} \}$ are linearly independent over $\mathbb{C}(q)$.
\end{lemma}

\begin{proof}
Take $n=a+b+c$, and define
\begin{itemize}
\item $S_1=\{k \in \mathbb{N} : 1\leq k \leq a+b \}$,
\item $S_2=\{k \in \mathbb{N} : a+b< k \leq a+2b+c \}$,
\item $S_3=\{k \in \mathbb{N} : a+2b+c < k \leq 2n \} $,
\item $S=\{p \in \ P_n: p \text{ has no pair in } S_i, ~i=1,2,3 \}$.
\end{itemize}
Note that for any pairing $p \in S$, $p$ has $a$ pairs between $S_1$ and $S_3$; $b$ pairs between $S_1$ and $S_2$; and $c$ pairs between $S_2$ and $S_3$.
So we obtain a bijection $\iota: T_{a+b} \times T_{b+c} \times T_{a+c} \to S$ via
\begin{align*}
\iota(p_1,p_2,p_3)=p_{a,b,c} \circ (p_1\otimes p_2 \otimes p_3),
\end{align*}
where $p_{a,b,c}$ is defined in Equation \eqref{Equ: p_{a,b,c}}, and $p_1\otimes p_2 \otimes p_3$ is a permutation on $2n$ points,
\begin{align*}
(p_1\otimes p_2 \otimes p_3)(k)
=\left\{
\begin{aligned}
&p_1(k), &&\forall~ 1\leq k \leq a+b, \\
&p_2(k-a-b)+a+b, &&\forall~ a+b < k \leq n+b, \\
&p_3(k-n-b)+n+b, &&\forall~ n+b < k \leq 2n. \\
\end{aligned}
\right.
\end{align*}
For any $p \in S$, we can choose $\hat{p} \in \mathcal{B}_n$ as
\[
\hat{p}=x_{\iota^{-1}(p)},
\]
where $x_{p_1,p_2,p_3}$ is defined in Equation \eqref{Equ: xp}.

Assume that
\[
\sum_{\nu} c_{\mu,\nu,\lambda} \rho_{\mu,\nu, \lambda} =0
\]
for some $c_{\mu,\nu,\lambda} \in \mathbb{C}(q)$.

Recall that  $\rho' \in \hom_{\sC}(\ty_{\mu} \otimes \ty_{1^{b}} \otimes \ty_{1^c} , \ty_{\lambda})$ and $\rho \in \hom_{\sC}(\ty_{\mu} \otimes \ty_{1^{b+c}}, \ty_{\lambda})$ are defined in Equations \eqref{Equ: rho} and \eqref{Equ: rho'}.
We identify the two hom spaces with subspaces of $\sC_{n}$.
By Proposition \ref{Prop: Bp Basis},
\begin{align*}
\rho_{\mu,\nu, \lambda} &=\sum_{p \in P_n} b_{\mu,\nu,\lambda}(p) \hat{p} \;; \\
\rho'_{\mu,\nu, \lambda} (\ty_{\mu} \otimes \ty_{1^{b+c},b}) &=\sum_{p \in P_n} b'_{\mu,\nu,\lambda}(p) \hat{p} \;,
\end{align*}
for some $b_{\mu,\nu,\lambda}(p), b'_{\mu,\nu,\lambda}(p) \in \mathbb{C}(q)$, and
\[
\sum_{\nu} c_{\mu,\nu,\lambda}  b_{\mu,\nu,\lambda}(p) =0, ~\forall p\in P_n \;.
\]

Take $S_0=\iota(T_{a+b} \times T_{b,c} \times T_{a+c})$.
Note that
\begin{align*}
\rho_{\mu,\nu,\lambda}=
\raisebox{-2.5cm}{
\begin{tikzpicture}
\draw[thick] (0,0) --++(0,2);
\draw[thick] (1,2) arc (-180:0:.5);
\draw[thick] (1,0) --++(2,2);
\node at (-.2,1) {$\ty_{\nu}$};
\node at (1-.2,1.8) {$\ty_{b}$};
\node at (1.7,1.8) {$\ty_{1^b}$};
\node at (1.7,1) {$\ty_{1^c}$};
\node at (2,2.5) {$\ty_{1^{b+c}}$};
\node at (0,3.2) {$\ty_{\mu}$};
\node at (0,2.5) {$\rho_{1,\nu}$};
\node at (0,-1.2) {$\ty_{\lambda}$};
\node at (0,-.5) {$\rho_{2,\nu}$};
\draw[thick] (2,2) --(2.5,2.5)--++(0,1);
\draw[thick] (3,2) --(2.5,2.5);
\draw[thick] (0,2) --(.5,2.5)--++(0,1);
\draw[thick] (1,2) --(.5,2.5);
\draw[thick] (0,0) --(.5,-.5)--++(0,-1);
\draw[thick] (1,0) --(.5,-.5);
\end{tikzpicture}
}
=
\raisebox{-2.5cm}{
\begin{tikzpicture}
\draw[thick] (0,0) --++(0,2);
\draw[thick] (1,2) arc (-180:0:.5);
\draw[thick] (1,0) --++(2,2);
\node at (-.2,1) {$a$};
\node at (1-.2,1.8) {$b$};
\node at (1.7,1) {$c$};
\node at (2,2.5) {$\ty_{1^{b+c}}$};
\node at (2,3.2) {$b+c$};
\node at (0,3.2) {$a+b$};
\node at (0,2.5) {$\rho_{1,\nu}$};
\node at (0,-1.2) {$a+c$};
\node at (0,-.5) {$\rho_{2,\nu}$};
\draw[thick] (2,2) --(2.5,2.5)--++(0,1);
\draw[thick] (3,2) --(2.5,2.5);
\draw[thick] (0,2) --(.5,2.5)--++(0,1);
\draw[thick] (1,2) --(.5,2.5);
\draw[thick] (0,0) --(.5,-.5)--++(0,-1);
\draw[thick] (1,0) --(.5,-.5);
\end{tikzpicture}
}
\end{align*}
On the other hand, $\ty_{1^{b+c},b}=\sum_{k}c_j  \hat{p}_{1,j} \otimes \hat{p}_{2,j}$ for some $c_j \in \mathbb{C}(q)$, $p_{1,j} \in T_{b}$ and $p_{2,j} \in T_{c}$ as defined in Equation \eqref{Equ: yi}. So
\begin{align*}
\rho'_{\mu,\nu, \lambda} (\ty_{\mu} \otimes \ty_{1^{b+c},b})
=\sum_{j}
\raisebox{-2.5cm}{
\begin{tikzpicture}
\draw[thick] (0,0) --++(0,2);
\draw[thick] (1,2) arc (-180:0:.5);
\draw[thick] (1,0) --++(2,2);
\node at (-.2,1) {$a$};
\node at (1-.2,1.8) {$b$};
\node at (1.7,1) {$c$};
\node at (2-.3,2.5) {$\hat{p}_{1,j}$};
\node at (3-.3,2.5) {$\hat{p}_{2,j}$};
\node at (1.8,3.2) {$b$};
\node at (2.8,3.2) {$c$};
\node at (0,3.2) {$a+b$};
\node at (0,2.5) {$\rho_{1,\nu}$};
\node at (0,-1.2) {$a+c$};
\node at (0,-.5) {$\rho_{2,\nu}$};
\draw[thick] (2,2) --++(0,1.5);
\draw[thick] (3,2) --++(0,1.5);
\draw[thick] (0,2) --(.5,2.5)--++(0,1);
\draw[thick] (1,2) --(.5,2.5);
\draw[thick] (0,0) --(.5,-.5)--++(0,-1);
\draw[thick] (1,0) --(.5,-.5);
\end{tikzpicture}
}
\end{align*}
Note that for any $p\in S_0$, if we express $\ty_{1^{b+c}}$ in terms of the basis $\mathcal{B}_{b+c}$, then only the components in $T_{b}\times T_{c}$  contribute non-zero coefficients of $p$.
Recall that $\ty_{1^{b+c},b}$ is the sum of such components of $\ty_{1^{b+c}}$ in $T_{b}\times T_{c}$, so
\begin{align*}
b'_{\mu,\nu,\lambda}(p)&= b_{\mu,\nu,\lambda}(p) \;, &&~\forall~ p\in S_0 \;, \\
b'_{\mu,\nu,\lambda}(p)&=0 \;, &&~\forall~ p \in S\setminus S_0 \;.
\end{align*}
Then
\[
\sum_{p\in S_0} \sum_{\nu} c_{\mu,\nu,\lambda}  b'_{\mu,\nu,\lambda}(p) \hat{p} =\sum_{p\in S_0} \sum_{\nu} c_{\mu,\nu,\lambda}  b_{\mu,\nu,\lambda}(p) \hat{p} =0 .
\]
Note that if $b'_{\mu,\nu,\lambda}(p) \neq 0$, $p$ has no pair in $S_1$ or $S_2$, then $p$ has no pair between the first $b$ points and the last $c$ points in $S_3$. So
\[
b'_{\mu,\nu,\lambda}(p) \ty_{\lambda}  \hat{p}  (\ty_{\mu} \otimes \ty_{1^b} \otimes \ty_{1^c}) \neq 0, \text{ only when } p \in S_0.
\]
By Lemma \ref{Lem: non-zero}, $\ty_{1^{b+c},b}(\ty_{1^b} \otimes \ty_{1^c})= c_0 \ty_{1^b} \otimes \ty_{1^c}$ for some $c_0 \neq 0$ in $\mathbb{C}(q)$. So
\begin{align*}
&c_0 \sum_{\nu} c_{\mu,\nu,\lambda}   \rho'(\tilde{\mu,\nu,\lambda}) \\
=&c_0 \sum_{\nu}  c_{\mu,\nu,\lambda}  \ty_{\lambda} \rho'(\tilde{\mu,\nu,\lambda}) (\ty_{\mu} \otimes \ty_{1^b} \otimes \ty_{1^c}) \\
=&\sum_{\nu}   c_{\mu,\nu,\lambda}  \ty_{\lambda} \rho'(\tilde{\mu,\nu,\lambda}) (\ty_{\mu} \otimes \ty_{1^{b+c},b})  (\ty_{\mu} \otimes \ty_{1^b} \otimes \ty_{1^c}) \\
=&\sum_{\nu} c_{\mu,\nu,\lambda} \ty_{\lambda} \left(  \sum_{p\in P_n}  b'_{\mu,\nu,\lambda}(p) \hat{p} \right) (\ty_{\mu} \otimes \ty_{1^b} \otimes \ty_{1^c}) \\
=& \sum_{\nu}  c_{\mu,\nu,\lambda}  \sum_{p\in S_0} b'_{\mu,\nu,\lambda}(p) \ty_{\lambda} \hat{p} (\ty_{\mu} \otimes \ty_{1^b} \otimes \ty_{1^c})  \\
=& \ty_{\lambda}   \left( \sum_{p\in S_0}  \sum_{\nu} c_{\mu,\nu,\lambda}  b'_{\mu,\nu,\lambda}(p) \hat{p}  \right) (\ty_{\mu} \otimes \ty_{1^b} \otimes \ty_{1^c}) \\
=&0 \;.
\end{align*}
By Lemma \ref{Lem: Basis 0}, we have that $c_{\mu,\nu,\lambda} =0$ for all $\nu$, Therefore, the elements $\{\rho'_{\mu,\nu, \lambda} : \nu \in \mu-b, \nu\in \lambda -1^{c} \}$ are linear independent in $\sC_n$.
In particular, $\rho'_{\mu,\nu, \lambda} \neq 0$, whenever $\nu \in \mu-b, \nu\in \lambda -1^{c}$.

We consider $\hat{p}$ as a morphism in $\hom_{\sC}(X^{n+b}, X^{a+c})$.
Then
\[
\ty_{\lambda} \left( \sum_{p\in S} \sum_{\nu} c_{\mu,\nu,\lambda}  b_{\mu,\nu,\lambda}(p) \hat{p} \right)  (\ty_{\mu} \otimes \ty_{1^b} \otimes \ty_{1^c})=0.
\]
Note that the set $S$ of pairings are the same as the pairings implemented by

\[
=\sum_{\nu}  c_{\mu,\nu,\lambda}  c_{r,i} \rho_{\nu,\lambda} (\ty_{\nu} \otimes m_{k} \otimes \ty_{1^{n-k}}) (\rho_{\mu,\nu} \otimes \ty_{1^i} \otimes \ty_{1^{r-i}})  \; .
\]
Note that $\rho_{\nu,\lambda} (\ty_{\nu} \otimes m_{k} \otimes \ty_{1^{n-k}}) (\rho_{\mu,\nu} \otimes \ty_{1^i} \otimes \ty_{1^{r-i}})\neq 0$  and they are linearly independent for different $\tau$.
Therefore,  $c_{\mu,\nu,\lambda}  c_{r,i}=0$. Recall that $c_{r,i} \neq 0$, so $c_{\mu,\nu,\lambda} =0$.
Therefore, the morphisms $\{\rho_{\mu,\nu, \lambda} : \nu \in \mu-b, \nu\in \lambda -1^{c} \}$ are linearly independent.
\end{proof}

\begin{theorem}\label{Thm: fusion with column}
Suppose $a,b,c \in \mathbb{N}$ and $r=b+c$. For any Young diagrams $\mu$ and $\lambda$, $|\mu|=a+b$, $|\lambda|=a+c$,
the elements $\{\rho_{\mu,\nu, \lambda} : \nu \in \mu-b, \nu\in \lambda -1^{c} \}$ form a basis of  $\hom_{\sC}(\ty_{\mu} \otimes \ty_{1^n}, \ty_{\lambda})$.
In particular, we obtain the fusion for $X_{1^r}$ in a closed form:
\[
X_{(1^r)} X_{\mu}=\sum_{i=0}^{r} \sum_{\nu \in \mu- i} \sum_{\lambda \in \nu+1^{r-i}}X_{\lambda}.
\]
\end{theorem}

\begin{proof}
By Lemmas \ref{Lem: Basis 1}, \ref{Lem: Basis 2}, $\{\rho_{\mu,\nu, \lambda} : \nu \in \mu-b, \nu\in \lambda -1^{c} \}$ form a basis of  $\hom_{\sC}(\ty_{\mu} \otimes \ty_{1^n}, \ty_{\lambda})$.
\end{proof}
We remove $i$ cells from $\mu$ (no two in the same column), and then we add $r-i$ cells (no two in the same row).

\begin{corollary}\label{Cor: fusion with row}
Applying the automorphism $\Omega$, we obtain the fusion with $X_{r}$ in a closed form:
\[
X_{(r)} X_{\mu}=\sum_{i=0}^{r} \sum_{\nu \in \mu- 1^i} \sum_{\lambda \in \nu+(r-i)}X_{\lambda}.
\]
\end{corollary}
We remove $i$ cells from $\lambda$ (no two in the same row), and then we add $n-i$ cells (no two in the same column).

\begin{remark}
The morphisms can be constructed explicitly following the construction in \cite{LiuYB}.
They are essentially used in the proof of Theorem \ref{Thm: fusion with column}.
We are going to compute the characters and the generating functions in the next section using Theorem \ref{Thm: fusion with column}.
\end{remark}

\section{Characters, Generating Functions and Fusion Rules for the Generic Case}\label{Sec: Characters, Generating Functions and Fusion Rules for the Generic Case}
We begin by introducing the tools we will need from the theory of symmetric functions. All the material we use can be found in the first chapter of \cite{Macdonald}.
\subsection{Symmetric Functions}
Recall that the ring of symmetric functions, $\Lambda$, is defined in the following way.
\begin{definition}
Let $n$ be a natural number, and $R_n = \mathbb{Z}[x_1, x_2, \ldots, x_n]^{S_n}$ be the ring of symmetric polynomial in $n$ variables. We write $R_n^k$ for the degree $k$ component of $R_n$. For each $k$, we have maps $\rho_n: R_n^k \to R_{n-1}^k$ defined by setting $x_n = 0$; these form an inverse system, so we may take the inverse limit $\displaystyle \lim_\leftarrow R_n^k$. Then, as an abelian group, we define:
\[
\Lambda = \bigoplus_{k \geq 0} \lim_\leftarrow R_n^k.
\]
The multiplication on $\Lambda$ is inherited from the multiplication $R_n^{k_1} \otimes R_n^{k_2} \to R_n^{k_1+k_2}$. We may complete $\Lambda$ with respect to the grading. In this case we obtain
\[
\hat{\Lambda} = \prod_{k \geq 0} \lim_\leftarrow R_n^k.
\]
\end{definition}
We introduce some important elements of the ring of symmetric functions.
\begin{proposition}
We have the following facts about $\Lambda$:
\begin{enumerate}
\item The polynomials $\displaystyle \sum_{i_1 < i_2 < \cdots < i_r} x_{i_1}x_{i_2} \cdots x_{i_r} \in R_n^{r}$ define an element $e_r \in \displaystyle \lim_\leftarrow R_n^k$ called the $r$-th elementary symmetric function. These $e_r$ freely generate $\Lambda$ as a polynomial ring: $\Lambda = \mathbb{Z}[e_1, e_2, \ldots]$. We have the generating function $\displaystyle E(t) = \sum_r e_r t^r = \prod_i (1 + x_it)$.
\item Similarly, the polynomials $\displaystyle \sum_{i_1 \leq i_2 \leq \cdots \leq i_r} x_{i_1}x_{i_2} \cdots x_{i_r} \in R_n^{r}$ define an element $h_r \in \displaystyle \lim_\leftarrow R_n^k$ called the $r$-th complete symmetric function. These $h_r$ also freely generate $\Lambda$ as a polynomial ring: $\Lambda = \mathbb{Z}[h_1, h_2, \ldots]$. We have the generating function $\displaystyle H(t) = \sum_r h_r t^r = \prod_i (1 - x_it)^{-1}$.
\item The polynomials $\displaystyle \sum_i x_i^r \in R_n^{r}$ define an element $p_r \in \displaystyle \lim_\leftarrow R_n^k$ called the $r$-th power-sum symmetric function. They freely generate $\mathbb{Q} \otimes \Lambda$ as a polynomial ring over $\mathbb{Q}$ (but they do not generate $\Lambda$ over $\mathbb{Z}$). We have the generating function $\displaystyle P(t) = \sum_r p_{r+1} t^r = \sum_i \frac{x_i}{1 - x_it}$.
\item The generating functions $E(t)$ and $H(t)$ satisfy the relation $H(t)E(-t) = 1$, and this equation encodes how to express the elementary symmetric functions in terms of the complete symmetric functions and vice versa. Similarly, we have $H^\prime(t)/H(t) = P(t)$, and $E^\prime(t)/E(t) = P(-t)$. In particular, we have the equations
\begin{eqnarray*}
\sum_{r \geq 0} h_r t^r &=& \exp \left( \sum_{i \geq 1} \frac{p_i}{i}t^i \right),\\
\sum_{r \geq 0} e_r t^r &=& \exp \left( \sum_{i \geq 1} \frac{(-1)^{i-1}p_i}{i}t^i \right).
\end{eqnarray*}
\item Elements of $\Lambda \otimes \Lambda$ may be viewed as polynomials in two sets of variables, say $x_i$ and $y_j$, symmetric in each separately. To indicate which variable set is being considered, we write $f(x)$ or $f(y)$. Given $f \in \Lambda$, we write $f(x, y)$ for the element of $\Lambda \otimes \Lambda$ defined by the symmetric function $f$ in the variable set $\{x_i\} \cup \{y_j\}$. (This operation defines a comultiplication $\Lambda \to \Lambda \otimes \Lambda$.)
\item Fix a Young diagram $\lambda = (\lambda_1, \lambda_2, \ldots)$, adding trailing zeros if needed, so that $\lambda$ has $n$ parts (usually we do not distinguish between Young diagrams that differ by trailing zeros). The polynomials $\det(x_i^{\lambda_j+n-j}) / \det(x_i^{n-j}) \in R_{n}^{|\lambda|}$ define an element $s_\lambda \in \displaystyle \lim_\leftarrow R_n^{|\lambda|}$, called the Schur function associated to $\lambda$.
\item We have that $s_{1^r} = e_r$ and $s_{r} = h_r$.
\end{enumerate}
\end{proposition}
\begin{example}
We have:
\[
\frac{p_1^2 + p_2}{2} = \frac{1}{2}\left(\sum_{i \neq j} x_i x_j + 2 \sum_i x_i^2\right) = \sum_{i \leq j}x_ix_j = h_2.
\]
\end{example}
\begin{remark}
Schur functions may be viewed as the characters of irreducible representations of $GL_n(\mathbb{C})$ in the following sense. If $M \in GL_n(\mathbb{C})$ has eigenvalues $x_i$, then the trace of the action of $M$ on the irreducible representation of $GL_n(\mathbb{C})$ corresponding to the Young diagram $\lambda$ is $s_\lambda(x_i)$: the Schur function corresponding to $\lambda$ evaluated at the eigenvalues $x_i$. Note that this quantity is zero unless the Young diagram $\lambda$ has at most $n$ nonzero parts. This means we have a homomorphism $\Lambda \to R_n$ whose kernel has basis $s_\mu$ for Young diagrams $\mu$ with more than $n$ nonzero parts.
\end{remark}
We now discuss a bilinear form on $\Lambda$.
\begin{proposition}
The ring $\Lambda$ satisfies the following properties:
\begin{enumerate}
\item The Schur functions form a $\mathbb{Z}$-basis of $\Lambda$: $\Lambda = \mathbb{Z}\{s_\lambda \mid \mbox{$\lambda$ a Young diagram}\}$. In particular, there is a bilinear form $\langle -, - \rangle : \Lambda \otimes \Lambda \to \mathbb{Z}$ for which the Schur functions are orthonormal.
\item The adjoint to multiplication by $p_i$ is $\displaystyle i \frac{\partial}{\partial p_i}$ (where elements of $\Lambda$ are viewed as polynomials in the $p_i$ with possibly rational coefficients).
\item The adjoint to multiplication by $s_\mu$ (with respect to $\langle -,- \rangle$) is denoted $s_\mu^\perp$. The symmetric function $s_\mu^\perp(s_\lambda)$ is called a skew-Schur function, and denoted $s_{\lambda / \mu}$. It is nonzero if and only if $\mu_i \leq \lambda_i$ for all $i$.
\item Schur functions satisfy the following multiplication rule $\displaystyle s_\mu s_\nu = \sum_{\mu, \nu} c_{\mu, \nu}^\lambda s_\lambda$, where $c_{\mu, \nu}^\lambda$ are the Littlewood-Richardson coefficients (which are zero unless $|\mu| + |\nu| = |\lambda|$). They also satisfy $\displaystyle s_\lambda(x, y) = \sum_{\mu, \nu} c_{\mu, \nu}^\lambda s_\mu(x)s_\nu(y)$.
\item The identity $e_r(x, y) = \sum_{i=0}^r e_i(x)e_{r-i}(y)$ shows that $c_{\mu, \nu}^{1^r}$ is zero unless $\mu = 1^i$ and $\nu = 1^{r-i}$ for some $0 \leq i \leq r$, in which case it is equal to 1.
\item The Littlewood-Richardson coefficient $c_{\mu, r}^\lambda$ is zero unless the diagram of $\lambda$ can be obtained by adding $r$ cells to the diagram of $\mu$, with no two cells in the same column; this is the Pieri rule. Similarly, $c_{\mu, 1^r}^\lambda$ is zero unless the diagram of $\lambda$ can be obtained by adding $r$ cells to the diagram of $\mu$, with no two cells in the same row; this is the dual Pieri rule.
\end{enumerate}
\end{proposition}
There are two identities that will be important to us, which we now state.
\begin{proposition}
We have the following equations:
\begin{enumerate}
\item The following equality of series holds in a completion of $R_n \otimes R_n$ for each $n$, and therefore in a completion of $\Lambda$ (note that the homogeneous components of the right-hand side define elements of the inverse limits used to define the ring of symmetric functions):
\[
\sum_{\lambda} s_\lambda(x) s_\lambda(y) = \prod_{i,j} \frac{1}{1-x_iy_j}.
\]
This is called the Cauchy Identity.
\item Similarly, we have the Dual Cauchy Identity:
\[
\sum_{\lambda} s_\lambda(x) s_{\lambda^\prime}(y) = \prod_{i,j} (1+x_iy_j).
\]
Here, $\lambda^\prime$ is the Young diagram dual to $\lambda$.
\end{enumerate}
\end{proposition}
There is another operation on symmetric functions called plethysm.
\begin{definition}
Given symmetric functions $f$ and $g$ which are sums of monomials in the variables $x_i$ with coefficients in $\mathbb{Z}_{\geq 0}$, the plethysm of $g$ with $f$ is a symmetric function denoted $g[f]$. It may be calculated in the following way. Express $f(x_1, x_2, \ldots )$ as a sum of monomials (repeated according to their multiplicity) $f = \sum_i x_1^{\alpha_1^{(i)}}x_2^{\alpha_2^{(i)}} \cdots $. Then $g[f]$ is the symmetric function obtained by evaluating $g$ on the variable set given by the monomials $x_1^{\alpha_1^{(i)}}x_2^{\alpha_2^{(i)}} \cdots $. It immediately follows that the map $\Lambda \to \Lambda$ defined by $g \mapsto g[f]$ is an algebra homomorphism (but this is not true for $f \mapsto g[f]$).
\end{definition}
\begin{remark}
There is a way of generalising the above definition to $f$ and $g$ for which are not necessarily a positive (or even integral, if one is prepared to base change $\Lambda$) sum of monomials. The most general definition is the one given in Chapter 1, Section 8 of \cite{Macdonald}.
\end{remark}
\begin{remark}
Let $f = \sum_{\mu} m_\mu s_{\mu}$ be the character of a representation $V$ of $GL_n(\mathbb{C})$ (where $n$ is taken to be sufficiently large), so $m_\mu \in \mathbb{Z}_{\geq 0}$, and all but finitely many $m_\mu$ are zero. Thus, $f$ encodes a homomorphism $\varphi_f: GL_n(\mathbb{C}) \to GL(V)$. Similarly, fix $\displaystyle g = \sum_{\nu} n_\nu s_\nu$ (with the same conditions on $n_\nu$ as on $m_\mu$), which uniquely defines a representation $W$ of $GL(V) = GL_{\dim(V)}(\mathbb{C})$, encoding a homomorphism $\varphi_g : GL(V) \to GL(W)$. Then, $W$ is a representation of $GL_n(\mathbb{C})$ via the composition $\varphi_g \circ \varphi_f$:
\[
GL_n(\mathbb{C}) \xrightarrow{\varphi_f} GL(V) \xrightarrow{\varphi_g} GL(W).
\]
The character of this representation is the plethysm $g[f]$. The value of $n$ used in this construction does not affect $g[f]$, provided it is large enough (e.g. $n=\deg(f) \deg(g)$ will suffice).
\end{remark}

\begin{example}
We show that $\displaystyle e_1 = \sum_{i} x_i$ is a two-sided identity for plethysm. Note that by definition, $e_1[f]$ recovers the sum of the monomials of $f$, namely $f$ itself. On the other hand, $f[e_1]$ is the evaluation of $f$ on the variable set $\{x_i\}$ (the monomials of $e_1$), which again is $f$ itself. This is consistent with the formulation in terms of representations of general linear groups, where $\varphi_{e_1}$ represents the identity map $GL_n(\mathbb{C}) \to GL_n(\mathbb{C})$ (for any $n$).
\end{example}

\begin{remark}\label{power_sum_plethysm_remark}
For power-sum symmetric functions $p_r$, plethysm has some useful properties. In particular, $p_r[f] = f[p_r]$ for arbitrary $f$, because both sides are equal to the symmetric function obtained by multiplying the exponents of all monomials of $f$ by $r$. As a special case, we obtain $p_{r_1}[p_{r_2}] = p_{r_2}[p_{r_1}] = p_{r_1r_2}$.
\end{remark}
Ultimately, the result we need about plethysm is the following.
\begin{theorem} \label{plethysm_theorem}
We have the following equation:
\[
h_r[h_2] = \sum_{|\lambda| = r} s_{2\lambda},
\]
where, if $\lambda = (\lambda_1, \lambda_2, \ldots, \lambda_k)$, then $2\lambda = (2\lambda_1, 2\lambda_2, \ldots, 2\lambda_k)$.
\end{theorem}
\begin{proof}
This can be found in Chapter 1, Section 8, Example 6 of \cite{Macdonald}. Alternatively, see Example $A2.9$ of \cite{Stanley}.
\end{proof}
We now introduce a linear operator which will play an important role in what follows, and prove some properties that it satisfies.
\begin{definition}
Let $L$ be the linear operator on the completion of $\Lambda$ defined by multiplication by $\displaystyle \prod_{i \leq j}\frac{1}{1 + x_ix_j}$.
\end{definition}

\begin{proposition} \label{adjoint_formula_prop}
The adjoint of $L$ with respect to $\langle -, - \rangle$ is:
\[
L^\dagger =  \sum_{\mu}(-1)^{|\mu|} s_{2\mu}^\perp.
\]
\end{proposition}

\begin{proof}
We recognise the product defining $L$ as the generating function of complete symmetric functions evaluated at $-1$, with variable set $\{x_{i}x_{j} \}_{i \leq j}$ (these are the monomials in $h_2$); the degree $2r$ component of this sum is precisely what is computed in Theorem \ref{plethysm_theorem}. Thus,
\[
\prod_{i \leq j}\frac{1}{1 + x_ix_j} = H(-1)[h_2] = \sum_{r \geq 0}(-1)^r h_r[h_2] = \sum_{r \geq 0} (-1)^r \sum_{|\mu| = r} s_{2 \mu} = \sum_{\mu}(-1)^{|\mu|} s_{2\mu}.
\]
Noting that the adjoint of multiplication by $s_{2\mu}$ is $s_{2\mu}^\perp$, the proposition follows.
\end{proof}

\begin{notation}\label{Not: Branching Coefficient}
Let $\phi_{2}:  GL_{n} \to GL_{n(n+1)/2}$ be the symmetric square representation of $GL_n$ and $\phi_{1^r}$ be the $r$-th antisymmetric power representation of $GL_{n(n+1)/2}$. Then $\phi_{1^r}\phi_{2}$ is a representation of $GL_n$. The multiplicity of the irreducible representation of $GL_n$ with highest weight $\lambda$ in $\phi_{1^r}\phi_{2}$ is denoted by $b_{n,r,\lambda}$.
We define $\displaystyle b_{r,\lambda}= \lim_{n \to \infty} b_{n,r,\lambda}$. Then
\begin{align}
e_r[h_2]&=\sum_{\lambda} b_{r,\lambda} s_\lambda  ;\\
\label{Equ: L-1}
L^{-1}&=\prod_{i \leq j} (1 + x_ix_j)=\sum_{r\geq 0} e_r[h_2] =\sum_{r\geq 0,\lambda} b_{r,\lambda} s_\lambda.
\end{align}
\end{notation}

\begin{lemma} \label{power_sum_L_lemma}
Let $\theta_i = \frac{1 + (-1)^i}{2}$, so that $\theta_i$ is equal to $0$ when $i$ is odd, and equal to $1$ when $i$ is even.
When expressed in terms of power-sum symmetric functions, $L$ has the following form:
\[
L = \exp \left( \sum_i \frac{(-1)^i p_i^2 + 2(-1)^{i/2} \theta_i p_{i}}{2i} \right).
\]
\end{lemma}
\begin{proof}
We write $L = H(-1)[h_2]$ (as in the proof of Proposition \ref{adjoint_formula_prop}), where we express $H(-1)$ and $h_2$ in terms of power-sum symmetric functions. We use Remark \ref{power_sum_plethysm_remark} to manipulate the plethysm:
\begin{eqnarray*}
L &=& \exp \left( \sum_i \frac{(-1)^ip_i}{i} \right)[\frac{p_1^2 + p_2}{2}] \\
&=& \exp \left( \sum_i \frac{(-1)^ip_i[\frac{p_1^2 + p_2}{2}] }{i} \right) \\
&=& \exp \left( \sum_i \frac{(-1)^i\frac{p_1^2 + p_2}{2}[p_i] }{i} \right) \\
&=& \exp \left( \sum_i \frac{(-1)^i\frac{p_i^2 + p_{2i}}{2}}{i} \right) \\
&=& \exp \left( \sum_i \frac{(-1)^i(p_i^2 + p_{2i})}{2i} \right).
\end{eqnarray*}
We rearrange the sum so that all instances of $p_i$ occur in the $i$-th summand. This means moving the term $(-1)^i p_{2i}/2i$ from the $i$-th summand to the $2i$-th summand. Upon noting that only even index summands obtain a contribution in this way, we obtain the stated formula.
\end{proof}

\begin{proposition} \label{comultiplication_L_prop}
Consider $\Lambda \otimes \Lambda$ as the set of symmetric functions in two sets of variables $\{x_i^{(1)}\}$ and $\{x_i^{(2)}\}$. Suppose that the symmetric function $f$ satisfies $f(x^{(1)}, x^{(2)}) = \sum_i g_i(x^{(1)})h_i(x^{(2)})$. Then, we have the following equation:
\[
\left( \sum_\lambda s_\lambda(x^{(1)})s_{\lambda^\prime}(x^{(2)})\right) L(f)(x^{(1)}, x^{(2)}) = L(g_i)(x^{(1)}) L(h_i)(x^{(2)}).
\]
\end{proposition}

\begin{proof}
In the definition of $L$ (considered to have variable set $\{x_i^{(1)}\} \cup \{ x_i^{(2)}\}$), products of pairs of variables take one of three forms: either both variables come from $\{x_i^{(1)}\}$, or both variables come from $\{x_i^{(2)}\}$, or one variable comes from each. Giving $L$ a subscript to show its variable set, we obtain:
\[
L_{\{x_i^{(1)}\} \cup \{ x_i^{(2)}\}} = \prod_{i_1 \leq i_2} \frac{1}{1 + x_{i_1}^{(1)}x_{i_2}^{(1)}} \prod_{i_1 \leq i_2} \frac{1}{1 + x_{i_1}^{(2)}x_{i_2}^{(2)}} \prod_{i_1, i_2} \frac{1}{1 + x_{i_1}^{(1)}x_{i_2}^{(2)}}.
\]
Moving the last factor to the left-hand side, and using the Dual Cauchy Identity,
\[
\left( \sum_\lambda s_\lambda(x^{(1)})s_{\lambda^\prime}(x^{(2)})\right) L_{\{x_i^{(1)}\} \cup \{ x_i^{(2)}\}} = L_{\{x_i^{(1)}\}}L_{\{x_i^{(2)}\}}.
\]
This is equivalent to the statement of the proposition.
\end{proof}

\subsection{Characters, Generating Functions and Fusion Rules}

In this section, we recall some properties of the Grothendieck ring $\mathcal{G}$, and then study its structure using symmetric functions.
Recall that $\mathcal{G}$ has basis $\{Y_\lambda\}$ indexed by Young diagrams.

\begin{notation}
By Schur-Weyl duality, we obtain a ring isomorphism $\Phi: \mathcal{G} \to \Lambda$, such that $\Phi(Y_{\lambda})=s_{\lambda}$. Moreover,
\[
Y_{\mu}Y_{\nu}=\sum_{\lambda} c_{\mu, \nu}^{\lambda}Y_{\lambda},
\]
where $c_{\mu, \nu}^{\lambda}$ are the Littlewood-Richardson coefficients.
\end{notation}

Recall that $s_{\lambda/2\mu}$ is a skew-Schur function, and $2\mu$ is the Young diagram obtained by doubling each part of $\mu$.

\begin{definition}\label{Def: Iso}
By Theorem \ref{Thm: ring iso X}, we define a ring isomorphism $\Phi: \mathcal{G} \to \Lambda$, such that
\[
\Phi(Y_{\lambda})=s_{\lambda}, ~\forall~ \lambda.
\]
In particular,
\[
\Phi(X_{1^r})=\Phi(Y_{1^r})=s_{1^r}, ~\forall~ r\geq 0.
\]
\end{definition}

\begin{theorem}\label{Thm: Main1}
For any Young diagram $\lambda$, we call $\Phi(X_\lambda)$ the character of $X_{\lambda}$ in $\mathcal{G}$. Then
\begin{align}
\Phi(X_\lambda) &= L^\dagger s_\lambda = \sum_{\mu} (-1)^{|\mu|} s_{\lambda/2\mu} \;; \\
\label{Equ: X=Y+2}
X_{\lambda}&=\sum_{\substack{\mu,\nu \\  2|\mu| +|\nu|= |\lambda|}}  (-1)^{|\mu|} c_{2\mu,\nu}^{\lambda} Y_{\nu} \;.
\end{align}
\end{theorem}

\begin{proof}
To prove the first statement, it suffices (by induction on $\lambda$) to show that the claimed expressions for the $\Phi(X_\lambda)$ multiply according to the rule defined by Theorem \ref{Thm: fusion with column}. When we encode the operations of removing and adding cells via the Pieri rule and dual Pieri rule, what we must prove becomes
\[
e_r L^\dagger (s_\lambda) = \sum_{i=0}^r L^\dagger (e_{r-i} h_i^\perp s_\lambda).
\]
This is precisely the assertion of the following equality of operators: $e_r L^\dagger = \sum_{i=0}^r L^\dagger e_{r-i} h_i^\perp$. We prove the adjoint of this equality, namely $L e_r^\perp = \sum_{i=0}^r h_i e_{r-i}^\perp L$. To prove this statement for all $r$ simultaneously, we multiply by $t^r$ and sum over $r \geq 0$; it is equivalent to prove the following identity of (operator-valued) generating functions:
\[
L E(t)^\perp = H(t)E(t)^\perp L.
\]
We rewrite all quantities in terms of power-sum symmetric functions. We have:
\begin{eqnarray*}
E(t)^\perp &=& \exp \left( \sum_i \frac{(-1)^{i-1}p_i^\perp}{i} t^i \right) = \exp \left( \sum_i (-1)^{i-1}\frac{\partial}{\partial p_i} t^i \right), \\
H(t) &=& \exp \left( \sum_i \frac{p_i}{i} t^i \right), \\
L &=& \exp \left( \sum_i \frac{(-1)^i p_i^2 + 2(-1)^{i/2} \theta_i p_{i}}{2i} \right).
\end{eqnarray*}
(Recall from Lemma \ref{power_sum_L_lemma} that $\theta_i$ is equal to $0$ if $i$ is odd, and equal to $1$ if $i$ is even.) We use an operator-theoretic version of Taylor's theorem, namely
\[
\exp \left(a \frac{\partial}{\partial x} \right) f(x) = f(x+a).
\]
Applying this termwise to the composition of operators $E(t)^\perp L$, we obtain:
\begin{eqnarray*}
E(t)^\perp L &=& \exp \left( \sum_i (-1)^{i-1}\frac{\partial}{\partial p_i} t^i \right)\exp \left( \sum_i \frac{(-1)^i p_i^2 + 2(-1)^{i/2} \theta_i p_{i}}{2i} \right) \\
&=& \exp \left( \sum_i \frac{(-1)^i (p_i + (-1)^{i-1}t^i)^2 + 2(-1)^{i/2} \theta_i (p_{i}+ (-1)^{i-1}t^i)}{2i} \right)\exp \left( \sum_i (-1)^{i-1}\frac{\partial}{\partial p_i} t^i \right) \\
&=& \exp \left( \sum_i \frac{(-1)^i p_i^2 -2t^ip_i + (-1)^i t^{2i} + 2(-1)^{i/2} \theta_i p_{i}+ 2(-1)^{i/2} \theta_i (-1)^{i-1}t^i}{2i} \right)E(t)^\perp \\
&=& \exp\left(-\sum_i \frac{p_i t^i}{i}\right) \exp \left(  \sum_i \frac{(-1)^i p_i^2 + 2(-1)^{i/2} \theta_i p_{i}}{2i} \right) \\
& & \times \exp \left( \sum_i \frac{(-1)^i t^{2i}+ 2(-1)^{i/2} \theta_i (-1)^{i-1}t^i}{2i} \right) E(t)^\perp .
\end{eqnarray*}
We recognise the first term as $H(t)^{-1}$, the second term as $L$, and the third term as $1$ (noting that all powers of $t$ cancel out). Thus we have:
\[
H(t)^{-1}LE(t)^\perp = E(t)^\perp L,
\]
which is equivalent to the statement
\[
\Phi(X_\lambda) = L^\dagger s_\lambda = \sum_{\mu} (-1)^{|\mu|} s_{\lambda/2\mu}.
\]
Furthermore,
\[
\Phi(X_{\lambda})=\sum_{\mu} (-1)^{|\mu|} s_{\lambda/2\mu}=\sum_{\substack{\mu,\nu \\  2|\mu| +|\nu|= |\lambda|}}  (-1)^{|\mu|} c_{2\mu,\nu}^{\lambda} s_{\nu}=\sum_{\substack{\mu,\nu \\  2|\mu| +|\nu|= |\lambda|}}  (-1)^{|\mu|} c_{2\mu,\nu}^{\lambda} \Phi(Y_{\nu}).
\]
Recall that $\Phi$ is an isomorphism, so
\[
X_{\lambda}=\sum_{\substack{\mu,\nu \\  2|\mu| +|\nu|= |\lambda|}}  (-1)^{|\mu|} c_{2\mu,\nu}^{\lambda} Y_{\nu},
\]
\end{proof}

\begin{theorem}\label{Thm: Y=X+2}
For a Young diagram $\lambda$, let us define $\lambda_{<}$ to be set of proper sub Young diagrams $\mu$, such that $|\lambda|-|\mu| \in 2\mathbb{N}^+$.
Then
\begin{align}\label{Equ: Y=X+2}
Y_{\lambda}&=X_{\lambda}+\sum_{\mu \in \lambda_{<}} n_{\lambda, \mu} X_{\mu}, \\
\sum_{\lambda} n_{\lambda, \mu} s_{\lambda}&=L^{-1}s_{\mu}= s_\mu \prod_{i \leq j} (1 + x_ix_j), \\
n_{\lambda, \mu}&=\sum_{r\geq0,\nu} b_{r,\nu} c_{\mu,\nu}^{\lambda}.
\end{align}

\end{theorem}

\begin{proof}
We assume that
\[
Y_{\lambda}=\sum_{\mu} n_{\lambda, \mu} X_{\mu}, \text{ for some } n_{\lambda, \mu} \in \mathbb{N}.
\]
By Theorem \ref{Thm: Main1},
\[
s_{\lambda}=\sum_{\mu} n_{\lambda, \mu} L^{\dagger}s_{\mu}.
\]
Then
\[
\langle L^{-1} s_\nu, s_{\lambda} \rangle =\sum_{\mu} n_{\lambda, \mu} \langle L^{-1}s_\nu,  L^{\dagger}s_{\mu} \rangle=\sum_{\mu} n_{\lambda, \mu} \langle s_\nu,  s_{\mu} \rangle=n_{\lambda, \nu}.
\]
By Equation \eqref{Equ: L-1},
\[
\sum_{\lambda} n_{\lambda, \mu} s_{\lambda}= \sum_{\lambda} \langle L^{-1} s_\mu, s_{\lambda} \rangle s_{\lambda}= L^{-1}s_{\mu}= s_\mu \prod_{i \leq j} (1 + x_ix_j).
\]
Moreover,
\[
n_{\lambda, \mu}=\langle L^{-1} s_\mu, s_{\lambda} \rangle =\langle \sum_{r\geq0, \nu} b_{r,\nu} s_\nu s_\mu, s_{\lambda} \rangle =\sum_{r\geq0,\nu} b_{r,\nu} c_{\mu,\nu}^{\lambda}.
\]

\end{proof}

\begin{theorem}\label{Thm: Main2}
We have the following generating function for $\Phi(X_{\lambda})$,
\[
\sum_{\lambda} s_\lambda(x) \Phi(X_\lambda)(y) = \prod_{i_1 \leq i_2}\frac{1}{1 + x_ix_j} \prod_{i, j}\frac{1}{1-x_i y_j}.
\]
(Here $\Phi(X_\lambda)(y)$ means that the symmetric function $\Phi(X_\lambda)$ has variable set $\{y_j \}$.)
\end{theorem}

\begin{proof}
Now we apply Theorem \ref{Thm: Main1} to prove this theorem.
We consider the first equation of Theorem \ref{Thm: Main1} as a having symmetric function variables $\{y_j\}$, and multiply by $s_\lambda(x)$. Summing over $\lambda$, we are required to show:
\[
\sum_{\lambda} s_\lambda(x) \sum_{\mu} (-1)^{|\mu|} s_{\lambda / 2\mu}(y) = \prod_{i_1 \leq i_2}\frac{1}{1 + x_ix_j} \prod_{i, j}\frac{1}{1-x_i y_j}
\]
We now calculate:
\begin{eqnarray*}
\sum_\lambda s_\lambda(x) \sum_\mu (-1)^{|\mu|} s_{\lambda / 2\mu}(y)
&=& \sum_\lambda s_\lambda(x) L^\dagger(s_\lambda)(y) \\
&=& \sum_\lambda \sum_\rho \langle s_\rho, L^\dagger(s_\lambda) \rangle s_\lambda(x) s_\rho(y) \\
&=& \sum_\rho \sum_\lambda \langle L (s_\rho), s_\lambda \rangle s_\lambda(x) s_\rho(y) \\
&=& \sum_\rho L (s_\rho)(x) s_\rho(y) \\
&=& \prod_{i_1 \leq i_2}\frac{1}{1 + x_ix_j} \sum_\rho s_\rho(x) s_\rho(y) \\
&=& \prod_{i_1 \leq i_2}\frac{1}{1 + x_ix_j} \prod_{i, j}\frac{1}{1-x_i y_j}.
\end{eqnarray*}
\end{proof}

\begin{theorem}\label{Thm: Main3}
We have the following fusion rules:
\[
R_{\mu, \nu}^ \lambda = \sum_{\alpha, \beta, \gamma} c_{\alpha, \beta}^\mu c_{\beta', \gamma}^\nu c_{\alpha, \gamma}^\lambda.
\]
(Here $\beta^\prime$ is the Young diagram dual to $\beta$.)
\end{theorem}

\begin{proof}
Now we apply Theorem \ref{Thm: Main2} to prove this theorem.
To do this, we consider a suitable generating function for the $R_{\mu, \nu}^\lambda$, and express it in terms of two instances of the generating function in the second part of the theorem. We work with three variable sets: $\{x_i^{(1)}\}$, $\{x_i^{(2)}\}$, and $\{y_j\}$, and use Proposition \ref{comultiplication_L_prop}.
\begin{eqnarray*}
& &\sum_{\mu, \nu, \lambda} R_{\mu, \nu}^{\lambda} s_\mu(x^{(1)}) s_\nu(x^{(2)}) \Phi(X_\lambda)(y) \\
&=& \sum_{\mu, \nu} s_\mu(x^{(1)}) s_\nu(x^{(2)}) \Phi(X_\mu)(y)  \Phi(X_\nu)(y) \\
&=& \sum_{\mu} s_\mu(x^{(1)})\Phi(X_\mu)(y) \sum_{\nu}  s_\nu(x^{(2)}) \Phi(X_\nu)(y) \\
&=& \sum_{\mu} L(s_\mu)(x^{(1)}) s_\mu(y) \sum_{\nu} L(s_\nu)(x^{(2)}) s_\nu(y) \\
&=& \sum_{\mu, \nu} L(s_\mu)(x^{(1)}) L(s_\nu)(x^{(2)}) \sum_{\lambda}c_{\mu,\nu}^\lambda s_\lambda(y) \\
&=& \sum_{\mu, \nu} L_{\{x^{(1)}\}}L_{\{x^{(2)}\}} s_\mu(x^{(1)}) s_\nu(x^{(2)}) \sum_{\lambda}c_{\mu,\nu}^\lambda s_\lambda(y) \\
&=& \sum_{\mu, \nu} \left( \sum_\beta s_\beta(x^{(1)})s_{\beta^\prime}(x^{(2)})\right) L_{\{x^{(1)}\}\cup \{x^{(2)}\}} s_\mu(x^{(1)}) s_\nu(x^{(2)}) \sum_{\lambda}c_{\mu,\nu}^\lambda s_\lambda(y) \\
&=&  \left( \sum_\beta s_\beta(x^{(1)})s_{\beta^\prime}(x^{(2)})\right) L_{\{x^{(1)}\}\cup \{x^{(2)}\}} \sum_{\lambda}s_\lambda(y) \sum_{\mu, \nu}c_{\mu,\nu}^\lambda s_\mu(x^{(1)}) s_\nu(x^{(2)}) \\
&=&  \left( \sum_\beta s_\beta(x^{(1)})s_{\beta^\prime}(x^{(2)})\right) L_{\{x^{(1)}\}\cup \{x^{(2)}\}} \sum_{\lambda}s_\lambda(y) s_\lambda(x^{(1)}, x^{(2)}) \\
&=&  \left( \sum_\beta s_\beta(x^{(1)})s_{\beta^\prime}(x^{(2)})\right) \sum_\lambda s_\lambda(x^{(1)}, x^{(2)}) \Phi(X_\lambda)(y). \\
\end{eqnarray*}
At this point, we may take the coefficient of $\Phi(X_\lambda)(y)$ (these form a basis of $\Lambda$) to deduce
\begin{eqnarray*}
\sum_{\mu, \nu} R_{\mu, \nu}^{\lambda} s_\mu(x^{(1)}) s_\nu(x^{(2)})
&=& \left( \sum_\beta s_{\beta}(x^{(1)})s_{\beta^\prime}(x^{(2)})\right) s_\lambda(x^{(1)}, x^{(2)}) \\
&=& \left( \sum_\beta s_{\beta}(x^{(1)})s_{\beta^\prime}(x^{(2)})\right) \sum_{\alpha, \gamma} c_{\alpha, \gamma}^\lambda s_\alpha(x^{(1)}) s_\gamma(x^{(2)}) \\
&=& \sum_{\mu, \nu} \sum_{\alpha, \beta, \gamma} c_{\alpha, \beta}^\mu c_{\beta^\prime, \gamma}^\nu c_{\alpha, \gamma}^\lambda s_\mu(x^{(1)}) s_\nu(x^{(2)}).
\end{eqnarray*}
Taking coefficient of $s_\mu(x^{(1)}) s_\nu(x^{(2)})$, we recover the formula for $R_{\mu, \nu}^\lambda$.
\end{proof}

Note that if $|\alpha|=a$, $|\beta|=b$, $|\gamma|=c$,
and
\[
c_{\alpha, \beta}^\mu c_{\beta', \gamma}^\nu c_{\alpha, \gamma}^\lambda \neq 0,
\]
then $|\mu|=a+b$, $|\nu|=b+c$, $|\lambda|=a+c$. Conversely, $a$, $b$, $c$ are determined by $|\mu|$, $|\nu|$, $|\lambda|$. Thus the equation in Theorem \ref{Thm: Main3} is a finite sum.

\begin{definition}
Recall that the simple objects $\ty_{\gamma}$ and $\ty_{\gamma'}$ are dual to each other in $\sC$.
We denote $\cup_{\beta}$ to be the evaluation map in the hom space $\hom_{\sC}(\ty_{\beta} \otimes \ty_{\beta^\prime}, \emptyset)$.
For Young diagrams $\mu$, $\nu$, $\lambda$, $\alpha$, $\beta$, $\gamma$, with $|\mu|=a+b$, $|\nu|=b+c$, $|\lambda|=a+c$, $|\alpha|=a$, $|\beta|=b$, $|\gamma|=c$, we define the triangle map $\bigtriangledown: \hom_{\sC}(\ty_{\mu}, \ty_{\alpha} \otimes \ty_{\beta'}) \otimes  \hom_{\sC}(\ty_{\mu}, \ty_{\beta} \otimes \ty_{\gamma}) \otimes  \hom_{\sC}(\ty_{\alpha} \otimes \ty_{\gamma}, \ty_{\lambda}) \to \hom_{\sC}(\ty_{\mu} \otimes \ty_{\nu}, \ty_{\lambda})$ as
\begin{align*}
\bigtriangledown(\rho_1 \otimes \rho_2 \otimes \rho_3)&= \rho_3( \ty_{\mu} \otimes \cup_{\beta} \otimes \ty_{\lambda}) (\rho_1\otimes \rho_2) \\
&=
\raisebox{-2.5cm}
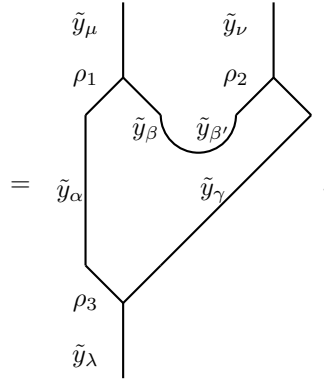
{
\begin{tikzpicture}
\draw[thick] (0,0) --++(0,2);
\draw[thick] (1,2) arc (-180:0:.5);
\draw[thick] (1,0) --++(2,2);
\node at (-.2,1) {$\ty_{\alpha}$};
\node at (1-.2,1.8) {$\ty_{\beta}$};
\node at (1.7,1.8) {$\ty_{\beta^\prime}$};
\node at (1.7,1) {$\ty_{\gamma}$};
\node at (2,3.2) {$\ty_{\nu}$};
\node at (2,2.5) {$\rho_2$};
\node at (0,3.2) {$\ty_{\mu}$};
\node at (0,2.5) {$\rho_{1}$};
\node at (0,-1.2) {$\ty_{\lambda}$};
\node at (0,-.5) {$\rho_{3}$};
\draw[thick] (2,2) --(2.5,2.5)--++(0,1);
\draw[thick] (3,2) --(2.5,2.5);
\draw[thick] (0,2) --(.5,2.5)--++(0,1);
\draw[thick] (1,2) --(.5,2.5);
\draw[thick] (0,0) --(.5,-.5)--++(0,-1);
\draw[thick] (1,0) --(.5,-.5);
\end{tikzpicture}
}.
\end{align*}
\end{definition}

\begin{theorem}\label{Thm: Basis generic}
For any Young diagrams  $\mu$, $\nu$, $\lambda$, the triangle map $\bigtriangledown$ is an embedding map and
\begin{align*}
\hom_{\sC}(\ty_{\mu} \otimes \ty_{\nu}, \ty_{\lambda})&=\bigoplus_{\alpha,\beta,\gamma}
\bigtriangledown\left(\hom_{\sC}(\ty_{\mu}, \ty_{\alpha} \otimes \ty_{\beta'}) \otimes  \hom_{\sC}(\ty_{\mu}, \ty_{\beta} \otimes \ty_{\gamma}) \otimes  \hom_{\sC}(\ty_{\alpha} \otimes \ty_{\gamma}, \ty_{\lambda}) \right).
\end{align*}
\end{theorem}
\begin{proof}
Similarly to the proof of Lemma \ref{Lem: Basis 1}, we
take \[
\tilde{x}_{p_1,p_2,p_3}=\ty_{\lambda} x_{p_1,p_2,p_3} (\ty_{\mu}\otimes \ty_{\nu}).
\]
Then
 $\{\tilde{x}_{p_1,p_2,p_3} : p_1 \in P_{a+b}, p_2 \in P_{b+c}, p_3 \in P_{a+c}.\}$ is a spanning set of $\hom_{\sC}(\ty_{\mu} \otimes \ty_{\nu}, \ty_{\lambda})$.
Note that the $180^{\circ}$ rotation of $s_{b}$ is $s_{b}$. So
\begin{align*}
\tilde{x}_{p_1,p_2,p_3}&=
\raisebox{-2.5cm}{
\begin{tikzpicture}
\draw[thick] (0,0) --++(0,2);
\draw[thick] (1,2) arc (-180:0:.5);
\draw[thick] (1,0) --++(2,2);
\node at (-.2,1) {$s_a$};
\node at (1-.2,1.8) {$s_b$};
\node at (1.7,1.8) {$s_b$};
\node at (1.7,1) {$s_c$};
\node at (2,3.2) {$\ty_{\nu}$};
\node at (2,2.5) {$\hat{p}_2$};
\node at (0,3.2) {$\ty_{\mu}$};
\node at (0,2.5) {$\hat{p}_1$};
\node at (0,-1.2) {$\ty_{\lambda}$};
\node at (0,-.5) {$\hat{p}_3$};
\draw[thick] (2,2) --(2.5,2.5)--++(0,1);
\draw[thick] (3,2) --(2.5,2.5);
\draw[thick] (0,2) --(.5,2.5)--++(0,1);
\draw[thick] (1,2) --(.5,2.5);
\draw[thick] (0,0) --(.5,-.5)--++(0,-1);
\draw[thick] (1,0) --(.5,-.5);
\end{tikzpicture}
}
=\sum_{\alpha,\beta,\gamma} \sum_{\rho_1,\rho_2,\rho_3}
\raisebox{-2.5cm}{
\begin{tikzpicture}
\draw[thick] (0,0) --++(0,2);
\draw[thick] (1,2) arc (-180:0:.5);
\draw[thick] (1,0) --++(2,2);
\node at (-.2,1) {$\ty_{\alpha}$};
\node at (1-.2,1.8) {$\ty_{\beta}$};
\node at (1.7,1.8) {$\ty_{\beta^\prime}$};
\node at (1.7,1) {$\ty_{\gamma}$};
\node at (2,3.2) {$\ty_{\nu}$};
\node at (2,2.5) {$\rho_2$};
\node at (0,3.2) {$\ty_{\mu}$};
\node at (0,2.5) {$\rho_{1}$};
\node at (0,-1.2) {$\ty_{\lambda}$};
\node at (0,-.5) {$\rho_{3}$};
\draw[thick] (2,2) --(2.5,2.5)--++(0,1);
\draw[thick] (3,2) --(2.5,2.5);
\draw[thick] (0,2) --(.5,2.5)--++(0,1);
\draw[thick] (1,2) --(.5,2.5);
\draw[thick] (0,0) --(.5,-.5)--++(0,-1);
\draw[thick] (1,0) --(.5,-.5);
\end{tikzpicture}
} \\
&=\sum_{\alpha,\beta,\gamma} \sum_{\rho_1,\rho_2,\rho_3}\bigtriangledown(\rho_1,\rho_2,\rho_3).
\end{align*}
for some Young diagrams $\alpha$, $\beta$, $\gamma$, with $|\alpha|=a$, $|\beta|=b$, $|\gamma|=c$, and some morphisms $\rho_{1}\otimes \rho_2 \otimes \rho_3 \in \hom_{\sC}(\ty_{\mu}, \ty_{\alpha} \otimes \ty_{\beta'}) \otimes  \hom_{\sC}(\ty_{\mu}, \ty_{\beta} \otimes \ty_{\gamma}) \otimes  \hom_{\sC}(\ty_{\alpha} \otimes \ty_{\gamma}, \ty_{\lambda})$.
Therefore
$$\bigcup_{\alpha,\beta,\gamma}
\bigtriangledown\left(\hom_{\sC}(\ty_{\mu}, \ty_{\alpha} \otimes \ty_{\beta'}) \otimes  \hom_{\sC}(\ty_{\mu}, \ty_{\beta} \otimes \ty_{\gamma}) \otimes  \hom_{\sC}(\ty_{\alpha} \otimes \ty_{\gamma}, \ty_{\lambda}) \right)$$
is a spanning set of $\hom_{\sC}(\ty_{\mu} \otimes \ty_{\nu}, \ty_{\lambda})$. By Proposition \ref{Prop: iso},
\begin{align*}
R_{\mu, \nu}^ \lambda &= \dim \hom_{\sC}(\ty_{\mu} \otimes \ty_{\nu}, \ty_{\lambda}) \\
&\leq \sum_{\alpha,\beta,\gamma} \dim \hom_{\sC}(\ty_{\mu}, \ty_{\alpha} \otimes \ty_{\beta'}) \times \dim  \hom_{\sC}(\ty_{\mu}, \ty_{\beta} \otimes \ty_{\gamma}) \times \dim \hom_{\sC}(\ty_{\alpha} \otimes \ty_{\gamma}, \ty_{\lambda})  \\
&= \sum_{\alpha,\beta,\gamma} \dim \hom_{H}(y_{\mu}, y_{\alpha} \otimes y_{\beta'}) \times \dim  \hom_{H}(y_{\mu}, y_{\beta} \otimes y_{\gamma}) \times \dim \hom_{H}(y_{\alpha} \otimes y_{\gamma}, y_{\lambda})  \\
&= \sum_{\alpha, \beta, \gamma} c_{\alpha, \beta}^\mu c_{\beta', \gamma}^\nu c_{\alpha, \gamma}^\lambda.
\end{align*}
By Theorem \ref{Thm: Main3}, the equality holds. So
the triangle map $\bigtriangledown$ is an embedding map and
\begin{align*}
\hom_{\sC}(\ty_{\mu} \otimes \ty_{\nu}, \ty_{\lambda})&=\bigoplus_{\alpha,\beta,\gamma}
\bigtriangledown\left(\hom_{\sC}(\ty_{\mu}, \ty_{\alpha} \otimes \ty_{\beta'}) \otimes  \hom_{\sC}(\ty_{\mu}, \ty_{\beta} \otimes \ty_{\gamma}) \otimes  \hom_{\sC}(\ty_{\alpha} \otimes \ty_{\gamma}, \ty_{\lambda}) \right).
\end{align*}
\end{proof}

\begin{remark}
Combining Theorem \ref{Thm: Basis generic} and Proposition \ref{Prop: iso}, we can construct an explicit basis of $\hom_{\sC}(\ty_{\mu} \otimes \ty_{\nu}, \ty_{\lambda})$ using $s_n$ and the basis of the hom spaces
$\hom_{H}(\ty_{\mu}, y_{\alpha} \otimes y_{\beta'})$, $\hom_{H}(y_{\mu}, y_{\beta}$, $ y_{\gamma})$, $\hom_{H}(y_{\alpha} \otimes y_{\gamma}, y_{\lambda})$ in the Hecke algebra $H$.
Applying the evaluation algorithm of the Yang-Baxter relation, we obtain the 6j-symbols of $\sC$. 

When the Young diagrams are small, the 6j-symbols can be computed by hand or by computer. We do not expect to compute 6j-symbols for large Young diagrams in this way, the complexity of this algorithm grows exponentially w. r. t. the size of the Young diagrams. Even computing the 6j-symbols for $Rep(H(q))$ remains challenging.
\end{remark}


\begin{thebibliography}{1}

\bibitem{Jon87}
V. F. R. Jones, {\em Hecke algebra representations of braid groups and link polynomials,} Annals of Math. 126 (1987) 335-388.

\bibitem{Jon99}
V. F. R. Jones, {\em Planar algebras, I} {https://arxiv.org/abs/math/9909027}{https://arxiv.org/abs/math/9909027}.


\bibitem{LiuYB}
Z.~Liu.
\newblock {\em Yang-baxter relation planar algebras}.
\href{https://arxiv.org/abs/1507.06030}{https://arxiv.org/abs/1507.06030},
\newblock 2015.


\bibitem{Macdonald}
I.G. Macdonald.
\newblock {\em Symmetric Functions and Hall Polynomials}.
\newblock Oxford classic texts in the physical sciences. Clarendon Press, 1998.



\bibitem{Stanley}
R.P. Stanley.
\newblock {\em Enumerative Combinatorics}.
\newblock Cambridge Studies in Advanced Mathematics. Cambridge University
  Press, 2001.


\bibitem{Xu}
F.~Xu, private discussions



\bibitem{Xu98}
F.~Xu \textit{New Braided Endomorphisms from Conformal Inclusions},
Communications in Mathematical Physics \textbf{192}.2 (1998), 249--403.







\end{thebibliography}
\end{document}